\documentclass[12pt,letterpaper]{amsart}
\usepackage{graphicx}
\usepackage{amsmath}
\usepackage{amssymb}
\usepackage{amsfonts}
\usepackage{amsthm}
\usepackage{cancel}
\usepackage[all]{xy}
\usepackage[neveradjust]{paralist}
\usepackage{enumitem}
\usepackage{multicol}
\usepackage{mathrsfs}
\usepackage{mathdots}
\usepackage[pagebackref,colorlinks=true,linkcolor=blue,citecolor=blue]{hyperref}

\usepackage{xcolor} 

\newcommand{\N}{\mathbb{N}}

\newcommand{\C}{\mathbb{C}}

\newcommand{\B}{{\mathcal B}}

\newcommand{\inv}{^{-1}}

\newcommand{\GL}{\mathrm{GL}}
\newcommand{\SO}{\mathrm{SO}}

\newcommand{\Sp}{\mathrm{Sp}}
\newcommand{\G}{\mathrm{G}}
\newcommand{\U}{\mathrm{U}}

\newcommand{\comment}[1]{}

\newtheorem{thm}{Theorem}[section]
\newtheorem{cor}[thm]{Corollary}
\newtheorem{lemma}[thm]{Lemma}
\newtheorem{prop}[thm]{Proposition}
\newtheorem {conj}[thm]{Conjecture}

\newtheorem {ques/conj}[thm]{Question/Conjecture}

\numberwithin{equation}{section}

\begin{document}

\title[Converse Theorem for Finite Split $\SO_{2l}$]{A Converse Theorem for Split $\SO_{2l}$ over Finite Fields}

\author{Alexander Hazeltine}
\address{Department of Mathematics\\
Purdue University\\
West Lafayette, IN, 47907, USA}
\email{ahazelti@purdue.edu}

\author{Baiying Liu}
\address{Department of Mathematics\\
Purdue University\\
West Lafayette, IN, 47907, USA}
\email{liu2053@purdue.edu}


\date{\today}

\dedicatory{}


\thanks{The research of the first-named author is partially supported by the NSF Grants DMS-1848058. 
The research of the second-named author is partially supported by the NSF Grants DMS-1702218, DMS-1848058, and by start-up funds from the Department of Mathematics at Purdue University.}

\begin{abstract}
    We prove a converse theorem for split even special orthogonal groups over finite fields. This is the only case left on converse theorems of split classical groups and the difficulty is the existence of the outer automorphism. In this paper, we develop new ideas and overcome this difficulty.
\end{abstract}

\maketitle


\section{Introduction}

Let $\G$ be a connected reductive group. 
In the representation theory of $\G$ over local fields $F$, converse theorems seek to uniquely identify a representation from its invariants. Among others, local gamma factors are important arithmetic invariants which play very important roles in the theory of Langlands program. 
More precisely, let $F$ be a nonarchimedean local field, and let $\pi$ be an irreducible generic representation of $\GL_n(F)$. The family of local twisted gamma factors $\gamma(s, \pi \times \tau, \psi)$, for $\tau$ any irreducible generic representation of $\GL_r(F)$, $\psi$ an additive character of $F$ and $s\in\mathbb{C}$, can be defined using Rankin--Selberg convolution \cite{JPSS83} or the Langlands--Shahidi method \cite{S84}. The local converse problem is to determine which family of local twisted gamma factors will uniquely determine $\pi$. The following is the famous Jacquet's conjecture on the local converse problem.

\begin{conj}[Jacquet's conjecture on the local converse problem]\label{lcp}
Let $\pi_1,\pi_2$ be irreducible generic representations
of $\GL_n(F)$. Suppose that they have the same central character.
If
\[
\gamma(s, \pi_1 \times \tau, \psi) = \gamma(s, \pi_2 \times \tau, \psi),
\]
as functions of the complex variable $s$, for all irreducible
generic representations $\tau$ of $\GL_r(F)$ with $1 \leq r \leq 
[\frac{n}{2}]$, then $\pi_1 \cong \pi_2$.
\end{conj}

Conjecture \ref{lcp} has recently been proved by Chai (\cite{Ch19}), and by Jacquet and the second-named author (\cite{JL18}), independently, using different analytic methods. Hence we have a local converse theorem for $\GL_n$. Local converse theorems for other classical groups also have been proved in recent years, mainly by Jiang-Soudry (\cite{JS03}, $\SO_{2n+1}$), by Zhang (\cite{Z18}, $\Sp_{2n}$, $\U_{2n}$, \cite{Z19}, $\U_{2n+1}$), and by Morimoto (\cite{Mor18}, $\mathrm{U}_{2n}$). For more references on local converse theorems, we refer to the introduction of \cite{LZ18}. 

Nien in \cite{Ni14} proved the finite fields analogue of Conjecture \ref{lcp} for cuspidal representations of $\GL_n$, using special properties of normalized Bessel functions and the twisted gamma factors defined by Roditty (\cite{Ro}). In \cite{LZ}, the second-named author and Zhang defined the twisted gamma factors for generic cuspidal representations of $\Sp_{2n}$, $\SO_{2n+1}$, $\U_{2n}$, and $\U_{2n+1}$, by proving certain multiplicity one results, and proved the corresponding converse theorems. 

The case left for the converse problems of classical groups over nonarchimedean local fields and finite fields is $\SO_{2n}$. The difficulty is the existence of the outer automorphism. In this paper and in \cite{HL22}, we develop new ideas and overcome this difficulty, for split $\SO_{2n}$ over finite fields and over nonarchimedean local fields, respectively. 
More precisely, in this paper, we define the twisted gamma factors $\gamma(\pi\times \tau, \psi)$ (see Proposition \ref{gammafactor}) for irreducible generic cuspidal representation $\pi$ of $\SO_{2l}(\mathbb{F}_q)$ and irreducible generic representation $\tau$ of $\GL_{n}(\mathbb{F}_q)$, 
and prove the following theorem. 

\begin{thm}[The Converse Theorem for $\SO_{2l}$, Theorem \ref{converse thm}]\label{converse thm intro}
Let $\pi$ and $\pi^\prime$ be irreducible cuspidal $\psi$-generic representations of split $\SO_{2l}(\mathbb{F}_q)$ with the same central character. If $$\gamma(\pi\times\tau,\psi)=\gamma(\pi^\prime\times\tau,\psi),$$ 
for all irreducible generic representations $\tau$ of $\GL_n(\mathbb{F}_q)$ with $n\leq l,$
then $\pi\cong\pi'$ or $\pi\cong c\cdot\pi',$
 where $c$ is the outer automorphism. 
\end{thm}

Theorem \ref{converse thm intro} implies that twisted gamma factors will not distinguish irreducible generic cuspidal representations $\pi$ and $c \cdot \pi$ of $\SO_{2l}(\mathbb{F}_q)$ (see also Corollaries \ref{conj n gamma} and \ref{conj l gamma}), which is a unique phenomenon for $\SO_{2l}$ among all the classical groups. This is consistent with the work of Arthur on the local Langlands correspondence and the local Langlands functoriality, and the work of Jiang and Soudry on local descent for $\SO_{2l}$ over nonarchimedean local fields (see \cite{Art} and \cite{JS12}). The analogue of Theorem \ref{converse thm intro} for quasi-split non-split $\SO_{2l}$ has more subtleties and has been proved by the first named author for both finite and local fields (\cite{Haz22a, Haz22b}). 

Now, we briefly introduce our new idea on proving Theorem \ref{converse thm intro}. 
The following is the key result.
As in other proven cases over finite fields, we make use of the normalized Bessel function $\mathcal{B}_{\pi,\psi}$ of $\pi$ which is a particular Whittaker function in the Whittaker model of $\pi$ (see Section \ref{bessel functions} for the definition). 
\begin{thm}[Theorem \ref{Bessels Equal}]\label{Bessels Equal intro}
Let $\pi$ and $\pi^\prime$ be irreducible cuspidal $\psi$-generic representations of split $\SO_{2l}(\mathbb{F}_q)$ with the same central character. If $$\gamma(\pi\times\tau,\psi)=\gamma(\pi^\prime\times\tau,\psi),$$
for all irreducible generic representations $\tau$ of $\GL_n$ with $1 \leq n\leq l,$
then we have that $$(\B_{\pi,\psi}+\B_{c\cdot\pi,\psi})(g)=(\B_{\pi^\prime,\psi}+\B_{c\cdot\pi^\prime,\psi})(g)$$  for any $g\in\SO_{2l}(\mathbb{F}_q)$.
\end{thm}

Our new idea is that for the $\GL_l$ twists, instead of considering the normalized Bessel function 
$\B_{\pi,\psi}$ and $\B_{c\cdot\pi,\psi}$ separately, we consider the summation of normalized Bessel functions 
$\B_{\pi,\psi}+\B_{c\cdot\pi,\psi}$.
More precisely, to prove Theorem \ref{Bessels Equal intro}, as in \cite{LZ}, we study the support of $\mathcal{B}_{\pi,\psi}$ on $\SO_{2l}(\mathbb{F}_q)$ and partition it based on 
Bruhat cells: $\mathrm{B}_n(\SO_{2l})$ for $n=1,\dots,l$ and $\mathrm{B}^c_l(\SO_{2l})$. 
Then we show that the twists by $\GL_n$ determine $\mathcal{B}_{\pi,\psi}$ and $\B_{c\cdot\pi,\psi}$ on Bruhat cells associated to Weyl elements in $\mathrm{B}_n(\SO_{2l})$ (see Theorem \ref{GL_n}), for $1 \leq n\leq l-2$, and the twists by $\GL_{l-1}$ determine $\mathcal{B}_{\pi,\psi}$ and $\B_{c\cdot\pi,\psi}$ on part of the Bruhat cells for those in $\mathrm{B}_{l-1}(\SO_{2l})$ (see Theorem \ref{GL_{l-1}}). However, the twists by $\GL_l$ determine the summation of the normalized Bessel functions 
$\B_{\pi,\psi}+\B_{c\cdot\pi,\psi}$ on 
the rest of the cells for $\mathrm{B}_{l-1}(\SO_{2l})$ and the cells of $\mathrm{B}_{l}(\SO_{2l})$ and $\mathrm{B}_{l}^c(\SO_{2l})$ (see Theorem \ref{GL_l}). Comparing to the case of $\Sp_{2l}(\mathbb{F}_q)$ in \cite{LZ} for example, where the authors show that the support of the normalized Bessel function $\mathcal{B}_{\pi,\psi}$ can be partitioned into $l$ sets, $\mathrm{B}_n(\Sp_{2l})$, $n=1,\ldots, l$, and for each $n$, the twists by $\GL_n$ determine exactly the normalized Bessel function on Bruhat cells associated to Weyl elements in $\mathrm{B}_n(\Sp_{2l})$.

Following is the structure of this paper. In Section 2, we introduce the groups and representations considered in this paper. In Section 3, we prove the multiplicity one results (Propositions \ref{MultOne} and \ref{Mult1}), and define the zeta integrals and gamma factors. In Section 4, we define Bessel functions and partition its support
(Proposition \ref{BesselPartition}). In Sections 5 - 7, we study the $\GL_n$ twists, $1 \leq n \leq l$, and show the relation between the $\GL_n$ twists and the support of Bessel functions (Theorems \ref{GL_n}, \ref{GL_{l-1}}, \ref{GL_l}). In Section 8, we prove Theorem \ref{Bessels Equal intro} and our main result Theorem \ref{converse thm intro}. 

\subsection*{Acknowledgements}

The authors would like to thank Professor Dihua Jiang and Professor Freydoon Shahidi for their interests and constant support.
The authors also would like to thank Professor Qing Zhang for helpful communications.

\section{The groups and representations}

Let $n, l \in \N$ and $q=p^r$ for some prime number $p\neq 2$. Let $\mathbb{F}_q$ be the finite field of $q$ elements and fix a nontrivial additive character $\psi$ on $\mathbb{F}_q$. Let $\GL_n$ to be the group of $n
\times n$ matrices with entries in $\mathbb{F}_q$ and non-zero determinant. Let $I_n$ be the identity element and define $J_n$ recursively by setting $J_1=1$ and $$J_n=
\left(\begin{matrix}
 0 & J_{n-1} \\
 1 & 0  
\end{matrix}\right).$$
We set $\SO_{n}=\{g\in\GL_{n} \, | \, \mathrm{det}(g)=1, {}^tgJ_{n} g = J_{n}\}$ to be the split special orthogonal groups.
Let $U_{\GL_{n}}$ and $U_{\SO_{2l}}$ be the subgroups of upper triangular unipotent matrices in $\GL_n$ and $\SO_{2l}$, respectively. Fix $B_{\SO_{2l}}=T_{\SO_{2l}}U_{\SO_{2l}}$ to be the Borel subgroup of $\SO_{2l}$ with split torus $T_{\SO_{2l}}.$

Set $$
c=\mathrm{diag}(I_{l-1},
\left(\begin{matrix}
 0 & 1 \\
 1 & 0  
\end{matrix}\right), I_{l-1}).$$ We have $c\notin \SO_{2l};$ however, $c \SO_{2l} c\inv=c\SO_{2l}c= \SO_{2l}$. Given a representation $\pi$ of $\SO_{2l}$ we define a new representation $c\cdot\pi$ of $\SO_{2l}$ by $c\cdot\pi(g)=\pi(cgc).$ Note that it is possible that $c\cdot\pi\cong\pi.$

We discuss the embeddings. These are the analogues for finite fields of the local cases found in \cite{K13, K15}. If $n<l$ we embed $\SO_{2n+1}$ into $\SO_{2l}$ via
$$
\left(
\begin{matrix}
A & B & C \\
D & E & K \\
L & P & Q 
\end{matrix} 
\right)
\mapsto
\mathrm{diag}(I_{l-n-1},
M^{-1}\left(\begin{matrix}
A & & B & C \\
 & 1 & & \\
D & & E & K \\
L & & P & Q 
\end{matrix}\right)
M, I_{l-n-1}),
$$
where $A$ and $Q$ are $n\times n$ matrices and 
$$
M=
\mathrm{diag}(I_n,\left(\begin{matrix}
    2 & -1 \\
    1 & \frac{1}{2}
    \end{matrix}\right), I_n).
$$
The embedding maps $\SO_{2n+1}$ into the standard Levi subgroup of $\SO_{2l}$ that is isomorphic to $\GL_{l-n-1}\times\SO_{2n+2}$.

If $n=l$, we embed $\SO_{2l}$ into $\SO_{2l+1}$ via
\begin{equation}\label{equ emb n=l}
   \left(\begin{matrix}
A & B \\
C & D
\end{matrix}\right)
\mapsto
M^{-1}
\left(\begin{matrix}
A & & B \\
 & 1 & \\
C & & D  
\end{matrix}\right)
M, 
\end{equation}
where $A, B, C,$ and $D$ are $l\times l$ matrices and
$
M=
\mathrm{diag}(I_{l-1},\tilde{M}, I_{l-1}),
$
where 
$$
\tilde{M}=\left(\begin{matrix}
    \frac{1}{4} & \frac{1}{2} & \frac{-1}{2} \\
    \frac{1}{2} & 0 & 1 \\
    \frac{-1}{2} & 1 & 1
    \end{matrix}\right).
$$
Note that the embedding takes the torus $T_{\SO_{2l}}$ to a torus in $\SO_{2l+1}$, but not the standard one consisting of diagonal matrices. Indeed, the embedding sends $t=\mathrm{diag}(t_1,\dots, t_l, t_l\inv,\dots,t_1\inv)$ to
$$
\mathrm{diag}(s,\left(\begin{matrix}
\frac{1}{2}+\frac{1}{4}(t_l+t_l\inv) & \frac{1}{2}(t_l-t_l\inv) & 2(\frac{1}{2}-\frac{1}{4}(t_l+t_l\inv)) \\
\frac{1}{4}(t_l-t_l\inv) & \frac{1}{2}(t_l+t_l\inv) & \frac{-1}{2}(t_l-t_l\inv) \\
\frac{1}{2}(\frac{1}{2}-\frac{1}{4}(t_l+t_l\inv)) & \frac{-1}{4}(t_l-t_l\inv) & \frac{1}{2}+\frac{1}{4}(t_l+t_l\inv)
\end{matrix}\right), s^*),
$$
where $s=\mathrm{diag}(t_1, t_2, \dots, t_{l-1})$ and $s^*=\mathrm{diag}(t_{l-1}\inv, \dots, t_2\inv, t_{1}\inv)$.

Next, we define generic characters and generic representations. Recall that $U_{\GL_n}$ and $U_{\SO_{2l}}$ are the subgroups of upper triangular unipotent matrices in $\GL_n$ and $\SO_{2l}$ respectively and that we fixed an additive nontrivial character $\psi$ of $\mathbb{F}_q$. We define a generic character, which by abuse of notation we still denote by $\psi$, on $U_{\GL_n}$ and $U_{\SO_{2l}}$. For $u=(u_{i,j})_{i,j=1}^n\in U_{\GL_n}$, we set $\psi(u)=\psi\left(\sum_{i=1}^{l-1} u_{i,i+1}\right).$ For $u=(u_{i,j})_{i,j=1}^l\in U_{\SO_{2l}}$ we set 
$$
\psi(u)= 
   \psi\left(\sum_{i=1}^{l-2} u_{i,i+1}+\frac{1}{4}u_{l-1,l}-\frac{1}{2} u_{l-1,l+1}  \right).
$$
We say an irreducible representation $\pi$ of $\SO_{2l}$ is $\psi$-generic if $$\mathrm{Hom}_{U_{\SO_{2l}}}(\pi,\psi)\neq 0.$$
Similarly, we say an irreducible representation $\tau$ of $\GL_{n}$ is $\psi$-generic if $$\mathrm{Hom}_{U_{\GL_{n}}}(\tau,\psi)\neq 0.$$
A nonzero intertwining operator in these spaces is called a Whittaker functional and it is well known that Whittaker functionals are unique up to scalars (by uniqueness of Whittaker models). 

Fix a nonzero Whittaker functional $\Gamma\in\mathrm{Hom}_{U_{\SO_{2l}}}(\pi,\psi)$. For $v\in\pi$, let $W_v(g)=\Gamma(\pi(g)v)$ for any $g\in\SO_{2l}$ and set $\mathcal{W}(\pi,\psi)=\{W_v \, | \, v\in\pi\}.$ $\mathcal{W}(\pi,\psi)$ is called the $\psi$-Whittaker model of $\pi.$ By Frobenius reciprocity, we have that $\mathrm{Hom}_{U_{\SO_{2l}}}(\pi,\psi)\cong\mathrm{Hom}_{\SO_{2l}}(\pi,\mathrm{Ind}_{U_{\SO_{2l}}}^{\SO_{2l}}(\psi)).$ Thus, $\pi$ can be realized as a subrepresentation of $\mathrm{Ind}_{U_{\SO_{2l}}}^{\SO_{2l}}(\psi)$ via the map $\pi \rightarrow \mathcal{W}(\pi,\psi)$ given by $v\mapsto W_v.$ Moreover, by uniqueness of Whittaker models, this subrepresentation occurs with multiplicity one inside $\mathrm{Ind}_{U_{\SO_{2l}}}^{\SO_{2l}}(\psi)$. We also note that the analogous results hold for $\psi$-generic representations $\tau$ of $\GL_n.$

Let  $Q_n=L_n V_n$ be the standard Siegel parabolic subgroup of $\SO_{2n+1}$ with Levi subgroup $L_n\cong \GL_n.$ For $a\in\GL_n$ we let $l_n(a)=\mathrm{diag}(a,1,a^*)\in L_n$ where $a^*=J_n{}^ta^{-1}J_n.$ Let $\tau$ be an irreducible generic representation of $\GL_{n}$ and set $I(\tau)=\mathrm{Ind}_{Q_n}^{\SO_{2n+1}}\tau.$ An element $\xi\in I(\tau)$ is a function $\xi:\SO_{2n+1}\rightarrow\tau$ satisfying $$
\xi(l_n(a)ug)=\tau(a)(\xi(g)), \forall a\in\GL_n, u\in V_n, g\in\SO_{2n+1}.
$$
Let $\Lambda_\tau \in \mathrm{Hom}_{U_{\GL_{n}}}(\tau,\psi^{-1})$ be a fixed nonzero homomorphism. For $\xi\in I (\tau)$, we define the function $f_\xi : \SO_{2n+1}\times\GL_n \rightarrow \C$ by $$
f_\xi(g,a)=\Lambda_\tau(\tau(a)\xi(g)).
$$
Let $I(\tau,\psi^{-1})$ be the space of functions generated by $f_\xi, \xi\in I(\tau).$ Note that for $f\in I(\tau,\psi^{-1}),$ we have 
$$
f(g,ua)=\psi^{-1}(u)f(g,a), \forall g\in\SO_{2n+1}, u\in U_{\GL_n}, a\in\GL_n.
$$
We  also let $\tau^*$ be the contragredient representation of $\GL_n$ defined by $\tau^*(a)=\tau(a^*).$

\section{Multiplicity one theorems and the gamma factor}

The goal of this section is to show that Bessel models for split even special orthogonal groups over finite fields are unique. Our primary reference for this is \cite{K15}, but the setup there is local. Note that the notation of this section (introduced below) agrees with the notation of \cite{GGP12b} with $H=N^{l-n}\times \SO_{2n+1}$, $\nu=\psi'$ in the case $n<l$ ($\psi'$ is extended to be trivial on the special orthogonal group), and $H=\SO_{2l}$, $\nu=1$ in the case $n=l$ .

We begin by considering the case $n=l.$ In this case, $H$ is $\SO_{2l}$ embedded inside of $\SO_{2l+1}$ as in \eqref{equ emb n=l}. Nonzero intertwining operators in $\mathrm{Hom}_{\SO_{2l}}(\pi,\tau)$ are called Bessel functionals which give Bessel models for $\pi$ (similar to the Whittaker models). 
The following proposition shows that the Bessel models are unique when $n=l$.

\begin{prop}\label{MultOne}
Let $Q=LV$ be a parabolic subgroup of $\SO_{2l+1}$ and $\sigma$ be an irreducible representation of $L.$ Let $\pi$ be an irreducible cuspidal represenation of $\SO_{2l}$ and $\tau=\mathrm{Ind}_Q^{\SO_{2l+1}} \sigma$. Then
$$
\mathrm{dim} \, \mathrm{Hom}_{\SO_{2l}}(\pi,\tau)\leq 1.
$$
\end{prop}

\begin{proof}
First, suppose that $\sigma$ is cuspidal. Then, \cite[Proposition 5.1]{GGP12a} still holds in our case. Indeed, we can follow their proof closely, except we use the multiplicity one results of \cite{AGRS, MW12}. This multiplicity one result was also known earlier in this case \cite{GPSR}. 

Suppose that $\sigma$ is not cuspidal. Then there exists a parabolic subgroup $Q^\prime=L^\prime V^\prime$ of $L$ and cuspidal representation $\sigma^\prime$ of $L'$ such that $\sigma\subseteq \mathrm{Ind}_{Q^\prime}^{L}\sigma^\prime.$ By transitivity of parabolic induction, $\tau\subseteq \mathrm{Ind}_{L^\prime V^\prime V}^{\SO_{2l+1}}\sigma^\prime\otimes 1_{ V}$ and hence the proposition follows from the case when $\sigma$ was cuspidal. 
\end{proof}

Next, we consider the case $n<l.$ We begin by defining the unipotent subgroup $N^{l-n}\subseteq \SO_{2l}$. Consider the standard parabolic subgroup $P_{l-n-1}=M_{l-n-1}N_{l-n-1}$ of $\SO_{2l}$ with Levi subgroup $M_{l-n-1}\cong\GL_{l-n-1}\times\SO_{2n+2}.$ Through this isomorphism, we embed $U_{\GL_{l-n-1}}$ inside of $\SO_{2l}$. We define $N^{l-n}=U_{\GL_{l-n-1}}N_{l-n-1}.$ That is,

$$
N^{l-n}=\left\{\left(\begin{matrix}
u_1 & v_1 & v_2 \\
 & I_{2n+2} & v_1^\prime \\
 & & u_1^*
\end{matrix}\right) \in \SO_{2l} \, | \, u_1\in U_{\GL_{l-n-1}} \right\}.
$$
For $v=(v_{i,j})\in N^{l-n}$ we define a character $\psi'$ of $N^{l-n}$ by
$$
\psi'(v)=
   \psi\left(\sum_{i=1}^{l-n-2} v_{i,i+1}+\frac{1}{4}v_{l-n-1,l}-\frac{1}{2} v_{l-n-1,l+1}   \right).
$$

Note that this character is trivial when $n=l-1.$ Let $H=\SO_{2n+1}N^{l-n}$, where $\SO_{2n+1}$ is realized via the embedding into $\SO_{2n+2}$ inside $M_{l-n-1}$. Extend $\psi'$ trivially across $\SO_{2n+1}$ so that $\psi'$ is a character of $H$. 
Nonzero intertwining operators in $\mathrm{Hom}_{H}(\pi,\tau\otimes \psi')$ are called Bessel functionals which give Bessel models for $\pi$. 
The below proposition gives the uniqueness of Bessel models for $n<l.$

\begin{prop}\label{Mult1}
Let $Q=LV$ be a parabolic subgroup of $\SO_{2n+1}$ and $\sigma$ be an irreducible representation of the Levi subgroup $L$. Let $\pi$ be an irreducible cuspidal representation of $\SO_{2l}$ and $\tau=\mathrm{Ind}_Q^{\SO_{2n+1}} \sigma.$ Then,
$$
\mathrm{dim}\, \mathrm{Hom}_{H}(\pi,\tau\otimes \psi')\leq 1.
$$
\end{prop}
\begin{proof}
The proof is the same as Proposition \ref{MultOne}, except instead of using the special orthogonal analogue of \cite[Proposition 5.1]{GGP12a}, we use the special orthogonal analogue of \cite[Proposition 5.3]{GGP12a}. Again, we also must substitute the appropriate multiplicity one results which follow from \cite{AGRS, MW12}.
 \end{proof}

\subsection{The Zeta Integrals}

Let $\pi$ be an irreducible $\psi$-generic cuspidal representation of $\SO_{2l}$ and $\tau$ be a generic representation of $\GL_n.$ Let $W\in\mathcal{W}(\pi,\psi)$ and $f\in I(\tau,\psi^{-1})$. Next, we shall define the zeta ``integrals" $\Psi(W,f)$ analogous to the local integrals of \cite{K15}. Note that \cite{K15} defines integrals for any $n$ and $l$; however, we  only need the case of $n\leq l$ for the converse theorem and so we do not consider the case of $n>l$. These integrals lie in the Bessel models and the uniqueness of these models give rise to the definition of the $\gamma$-factors.

First, suppose that $n=l.$ Then we define
$$
\Psi(W,f)=\sum_{g\in U_{\SO_{2l}}\setminus\SO_{2l}} W(g)f(w_{l,l}g, I_{l}),
$$
where 
$$
w_{l,l}=\left(\begin{matrix}
 \frac{1}{2}\cdot I_{l} & &  \\
  & 1 & \\
  & & 2\cdot I_l
\end{matrix}
\right)\in\SO_{2l+1}.$$
The integral satisfies the property $\Psi(g\cdot W,g\cdot f)=\Psi(W,f)$ for any $g\in\SO_{2l}.$

Next, suppose that $n<l.$ Then we define

$$
\Psi(W,f)=\sum_{g\in U_{\SO_{2n+1}}\setminus\SO_{2n+1}} \left(\sum_{r\in R^{l,n}} W(r w^{l,n} g (w^{l,n})\inv)\right)   f(g, I_{l}),
$$
where
$$
w^{l,n}=\left(\begin{matrix}
 & I_n & & & \\
 I_{l-n-1} & & & & \\
 & & I_2 & & \\
 & & & & I_{l-n-1} \\
 & & & I_n &
\end{matrix}
\right)\in\SO_{2l},$$
and
$$
R^{l,n}=\left\{\left(\begin{matrix}
I_n & & & & \\
x & I_{l-n-1} & & & \\
& & I_2 & & \\
& & & I_{l-n-1} & \\
& & & x^\prime & I_n
\end{matrix}\right)\in\SO_{2l}\right\}.
$$
The integral satisfies the property $\Psi((gn)\cdot W,g\cdot f)=(\psi')\inv(n)\Psi(W,f)$ for any $g\in\SO_{2n+1}$ and $n\in N^{l-n}.$ Note that in the case $n<l$, our integral differs from \cite{K15} slightly. The difference is a right translation of the Whittaker function by $(w^{l,n})\inv.$ 

Let $n\leq l.$ We define an intertwining operator $M(\tau,\psi^{-1}): I(\tau,\psi^{-1})\rightarrow I(\tau^*,\psi^{-1})$ by 
$$
M(\tau,\psi^{-1})f(h,a)=\sum_{u\in V_n} f(w_n u h, d_n a^*),
$$
where $w_n=\left(\begin{matrix}
 & & I_n \\
 & (-1)^n & \\
 I_n & &
\end{matrix}\right)$, $d_n=\mathrm{diag}(-1,1,-1,\dots,(-1)^n)\in\GL_n$, $a^*=J_n {}^t a^{-1} J_n$, and $V_n$ is the unipotent radical of the standard Siegel parabolic subgroup $Q_n=L_n V_n$ in $\SO_{2n+1}$ with $L_n\cong \GL_n$.

The following proposition gives the definition of the $\gamma$-factor.

\begin{prop}\label{gammafactor}
Let $\pi$ be an irreducible $\psi$-generic cuspidal representation of $\SO_{2l}$, $\tau$ be a generic representation of $\GL_n$, $W\in\mathcal{W}(\pi,\psi)$, and $f\in I(\tau,\psi^{-1})$. Then there exists a complex constant $\gamma(\pi\times\tau, \psi)$, called the $\gamma$-factor of $\pi$ and $\tau$, such that
$$
\gamma(\pi\times\tau, \psi)\Psi(W,f)=\Psi(W,M(\tau,\psi^{-1})f).
$$
\end{prop}

\begin{proof}
This is immediate from Propositions \ref{MultOne} and \ref{Mult1}.
 \end{proof}

We refer to these integrals and $\gamma$-factors as the twists by $\GL_n.$

\section{Bessel Functions}\label{bessel functions}

In this section, we define Bessel functions and study their properties. The Bessel functions are particular functions in the Whittaker model of a $\psi$-generic representation. The analysis of the zeta integrals evaluated on Bessel functions is crucial in our proof of the converse theorem.

Consider $\pi$ as a representation of $B_{\SO_{2l}}$. Let $\pi(U_{\SO_{2l}}, \psi)$ be the subspace of $\pi$ generated by $$\{\pi(u)v-\psi(u)v \, | \, u\in U_{\SO_{2l}}, \, v\in\pi\},$$ and let $\pi_{U_{\SO_{2l}}, \psi}=\pi/\pi(U_{\SO_{2l}}, \psi)$ be the twisted Jacquet module. Since $\pi$ is an irreducible $\psi$-generic representation, $\mathrm{dim}(\pi_{U_{\SO_{2l}}, \psi})=1.$ Let $v\in\pi$ with $v\notin \pi(U_{\SO_{2l}}, \psi)$ and set $$
v_0=\frac{1}{|U_{\SO_{2l}}|}\sum_{u\in U_{\SO_{2l}}}\psi^{-1}(u)\pi(u)v.
$$
By the Jacquet-Langlands lemma \cite[Lemma 2.33]{BZ76}, $v_0\neq 0.$ Recall that we fixed a nonzero Whittaker functional  $\Gamma\in\mathrm{Hom}_{U_{\SO_{2l}}}(\pi,\psi)$. Then we have $\Gamma(v_0)\neq 0$ and by construction $\pi(u)v_0=\psi(u)v_0$ for any $u\in U_{\SO_{2l}}.$ Any such vector is called a Whittaker vector of $\pi.$
For $g\in\SO_{2l}$, set $$\mathcal{B}_{\pi,\psi}(g)=\frac{\Gamma(\pi(g)v_o)}{\Gamma(v_0)}.$$ $\mathcal{B}_{\pi,\psi}$ is called the normalized Bessel function for $\pi$ and it is simple to verify that $\mathcal{B}_{\pi,\psi}\in\mathcal{W}(\pi,\psi).$ The following proposition follows from the definitions.

\begin{prop}\label{Besselprop}
$\mathcal{B}_{\pi,\psi}(I_{2l})=1$ and $\mathcal{B}_{\pi,\psi}(u_1gu_2)=\psi(u_1u_2)\mathcal{B}_{\pi,\psi}(g)$ for any $g\in\SO_{2l}$ and any $u_1,u_2\in U_{\SO_{2l}}.$
\end{prop}

Let $W(\SO_{2l})$ be the Weyl group of $\SO_{2l}$ and let $\Delta(\SO_{2l})$ be the set of simple roots. We say that a Weyl element $w\in W(\SO_{2l})$ {\it supports Bessel functions} if for any $\alpha\in\Delta(\SO_{2l})$, $w\alpha$ is either negative or simple. We let B($\SO_{2l}$) denote the set of Weyl elements which support Bessel functions. We call B$(\SO_{2l})$ the Bessel support. Recall $B_{\SO_{2l}}=T_{\SO_{2l}} U_{\SO_{2l}}$ is a fixed Borel subgroup of $\SO_{2l}.$ The following lemma justifies our terminology. 

\begin{lemma}
Let $\pi$ be an irreducible $\psi$-generic representation of $\SO_{2l}$ with Bessel function $\mathcal{B}_{\pi,\psi}$. Then, for any $w\in W(\SO_{2l})\setminus\mathrm{B}(\SO_{2l}),$ we have $\mathcal{B}_{\pi,\psi}(g)=0,$ for any $g\in B_{\SO_{2l}} w B_{\SO_{2l}}.$
\end{lemma}

\begin{proof}
Since $B_{\SO_{2l}}=T_{\SO_{2l}}U_{\SO_{2l}}=U_{\SO_{2l}}T_{\SO_{2l}}$ as a set and by Proposition \ref{Besselprop},  $$\mathcal{B}_{\pi,\psi}(g)=\mathcal{B}_{\pi,\psi}(t_1u_1wt_2u_2)=\mathcal{B}_{\pi,\psi}(u_1^\prime t_1^\prime w t_2 u_2)=\psi(u_1^\prime u_2)\mathcal{B}_{\pi,\psi}(t_1^\prime w t_2),$$ for some $u_1,u_1^\prime,u_2\in U_{\SO_{2l}}$ and $t_1,t_1^\prime, t_2\in T_{\SO_{2l}}.$ Thus it is enough to show that $\mathcal{B}_{\pi,\psi}(t_1 w t_2)=0$ for any $t_1,t_2\in T_{\SO_{2l}}.$ By definition of the Weyl group, $wt_2=t_2^\prime w$ for some $t_2^\prime \in T_{\SO_{2l}}.$ Hence, $\mathcal{B}_{\pi,\psi}(t_1 w t_2)=\mathcal{B}_{\pi,\psi}(t_1 t_2^\prime w)$ and so it is enough to show that $\mathcal{B}_{\pi,\psi}(t w)=0$ for any $t\in T_{\SO_{2l}}.$

Since $w\notin \mathrm{B}(\SO_{2l}),$ there exists $\alpha\in\Delta(\SO_{2l})$ such that $w\alpha$ is positive but not simple. Let $x\in\mathbb{F}_q$ and let $\mathrm{\bold{x}}_\alpha(x)$ be an element in the root space of $\alpha$ (note that the root space of $\alpha$ lies in $U_{\SO_{2l}}$ and is isomorphic to $\mathbb{F}_q$). Then $tw\mathrm{\bold{x}}_\alpha(x) = \mathrm{\bold{x}}_{w\alpha}(x^\prime)tw$ for some $x^\prime\in\mathbb{F}_q$ (because $T_{\SO_{2l}}$ normalizes $U_{\SO_{2l}}$). Since the simple root spaces are exactly the support of $\psi$, $\psi(\mathrm{\bold{x}}_\alpha(x))$ is a nonzero constant multiple of $\psi(x)$ and this constant is independent of $x$. Since $w\alpha$ is positive but not simple, $\psi(\mathrm{\bold{x}}_{w\alpha}(x^\prime)))=\psi(0)=1.$ Thus, by Proposition \ref{Besselprop}, 
$\psi(x)\mathcal{B}_{\pi,\psi}(t w)=\mathcal{B}_{\pi,\psi}(t w)$ 
for any $x\in\mathbb{F}_q$. Since $\psi$ is nontrivial, $\mathcal{B}_{\pi,\psi}(t w)=0.$ This proves the lemma.
 \end{proof}

Next, we determine the support of a Bessel function on the torus.

\begin{lemma}\label{center}
If $t\in T_{\SO_{2l}}$ and $\B_{\pi,\psi}(t)\neq 0,$ then $t$ is in the center of $\SO_{2l}.$
\end{lemma}

\begin{proof}
Let $t\in T_{\SO_{2l}}$ such that $\B_{\pi,\psi}(t)\neq 0$ and let $\delta$ be any simple root of $\SO_{2l}.$ For $x\in\mathbb{F}_q$, let $\bold{x}_\delta(x)$ be in the root subgroup of $\delta$. Then $t \bold{x}_\delta(x) = \bold{x}_\delta(\delta(t) x) t.$ Hence, by Proposition \ref{Besselprop}, $$
\psi(\bold{x}_\delta(x))\B_{\pi,\psi}(t) = \psi(\bold{x}_\delta(\delta(t)x))\B_{\pi,\psi}(t).
$$
Thus, $\psi(\bold{x}_\delta(x))= \psi(\bold{x}_\delta(\delta(t)x))$ for any $x\in\mathbb{F}_q$ and simple root $\delta\in\Delta(\SO_{2l}).$ Since $\psi$ is nontrivial and $x$ is arbitrary, we must have $\delta(t)=1$ for all $\delta\in\Delta(\SO_{2l}).$ Thus, $t$ is in the center of $\SO_{2l}$ which proves the lemma.
 \end{proof}

\subsection{Bessel Function of the Conjugate representation}

In this section, we determine the normalized Bessel function in the $\psi$-Whittaker model of the conjugate representation $c\cdot\pi.$

Let $\psi_c$ be the character on $U_{\SO_{2l}}$ defined by $\psi_c(u)=\psi(cuc).$ Recall that we fixed a nonzero Whittaker functional $\Gamma\in\mathrm{Hom}_{U_{\SO_{2l}}}(\pi,\psi).$ It follows that $\Gamma\in\mathrm{Hom}_{U_{\SO_{2l}}}(c\cdot\pi,\psi_c)$ and $$\mathcal{B}_{c\cdot\pi,\psi_c}(g)=\frac{\Gamma(c\cdot\pi(g)v_o)}{\Gamma(v_0)}=\mathcal{B}_{\pi,\psi}(cgc)$$ defines the normalized Bessel function for $c\cdot\pi$ in $\mathrm{Ind}_{U_{\SO_{2l}}}^{\SO_{2l}}\psi_c$. On the other hand, $c\cdot\pi$ is also $\psi$-generic, since $\psi_c(\tilde{t}\inv u\tilde{t})=\psi(u)$ for any $u\in U_{\SO_{2l}},$ where
\begin{equation}\label{def tildet}
  \tilde{t}=\mathrm{diag}(I_{l-1},\frac{-1}{2},-2,I_{l-1})\in T_{\SO_{2l}}.  
\end{equation}
Let
$$
\mathcal{B}_{c\cdot\pi, \psi}(g):=\mathcal{B}_{c\cdot\pi, \psi_c}(\tilde{t}\inv g \tilde{t})=
\mathcal{B}_{\pi, \psi}(c\tilde{t}\inv g \tilde{t}c).
$$
Then we have the following. 

\begin{prop}
$\mathcal{B}_{c\cdot\pi, \psi}$ is the normalized Bessel function for $c\cdot\pi$ in $\mathrm{Ind}_{U_{\SO_{2l}}}^{\SO_{2l}}\psi.$
\end{prop}

\begin{proof}
Let $\Gamma'$ be defined by $\Gamma'(v)=\Gamma(c\cdot\pi(\tilde{t}\inv)v).$ Then,
$\Gamma' \in \mathrm{Hom}_{U_{\SO_{2l}}}(c\cdot\pi,\psi)$. Since $\tilde{t}c=c\tilde{t}\inv$, $\mathcal{B}_{c\cdot\pi, \psi}$ is the right translation by $\tilde{t}\inv$ of a Whittaker function of $c\cdot \pi$ in $\mathrm{Ind}_{U_{\SO_{2l}}}^{\SO_{2l}}\psi$. Thus, $\mathcal{B}_{c\cdot\pi, \psi}$ is a Whittaker function of $c\cdot\pi$.

Next, $\mathcal{B}_{c\cdot\pi, \psi}(I_{2l})=\mathcal{B}_{\pi, \psi}(I_{2l})=1$ and for $u_1, u_2\in U_{\SO_{2l}}$ and $g\in\SO_{2l}$
\begin{align*}
    \mathcal{B}_{c\cdot\pi, \psi}(u_1gu_2)&=\mathcal{B}_{c\cdot\pi, \psi_c}(\tilde{t}\inv u_1 \tilde{t} \tilde{t}\inv g \tilde{t} \tilde{t}\inv u_2 \tilde{t})\\
    &=\psi_c(\tilde{t}\inv u_1 \tilde{t})\psi_c(\tilde{t}\inv u_2 \tilde{t})\mathcal{B}_{c\cdot\pi, \psi_c}(\tilde{t}\inv g\tilde{t})\\
&=\psi(u_1)\psi(u_2)\mathcal{B}_{c\cdot\pi, \psi}(g).
\end{align*}
By Proposition \ref{Besselprop}, $\mathcal{B}_{c\cdot\pi, \psi}$ is the normalized Bessel function of $c\cdot\pi$ which proves the proposition.
 \end{proof}

\subsection{Partition of the Bessel Support}\label{BesselSupport}

In this section we aim to partition the Bessel support and provide some preparation for future calculations. We  define the Weyl elements that appear in the computations of the zeta integrals and then show that these can be used to partition the Bessel support. When $n<l,$ we embed $w_n$ into $\SO_{2l}$ and use the image to define the Weyl elements we are interested in. When $n=l,$ we take a different approach since the embedding is $\SO_{2l}\hookrightarrow\SO_{2l+1}$ and $w_l\in\SO_{2l+1}.$ 

Let $n<l$ and recall that 
$w_n=\left(\begin{matrix}
 & & I_n \\
 & (-1)^n & \\
 I_n & &
\end{matrix}\right) \in\SO_{2n+1}$. Let 
$$
\hat{w}_n=\left\{\begin{array}{cc}
  \left(\begin{matrix}
  I_{l-n-1} & & & & & \\
   & & & & I_n & \\
   & & & 1 &  & & \\
   & & 1 &  & & & \\
   & I_n & & & & & \\
   & & & & & & I_{l-n-1} \\
  \end{matrix}\right)  & \mathrm{if} \, n \, \mathrm{is} \,  \mathrm{odd,} \\
   \left(\begin{matrix}
  I_{l-n-1} & & & & & \\
   & & & & I_n & \\
   & & 1 &  &  & & \\
   & &  & 1 & & & \\
   & I_n & & & & & \\
   & & & & & & I_{l-n-1} \\
  \end{matrix}\right)  & \mathrm{if} \, n \, \mathrm{is} \,  \mathrm{even.} \\
\end{array}\right. 
$$
Then the image of $w_n$ in $\SO_{2l}$ is $t_n^\prime \hat{w}_n$, where $t_n^\prime=\tilde{t}$ as in \eqref{def tildet} if $n$ is odd, and $t_n^\prime=I_{2l}$ if $n$ is even. Note that if $n$ is odd, then $c w_n c=\tilde{t}\inv\hat{w}_n$, and if $n$ is even then $c w_n  c  = \hat{w}_n$ (here we realize $cw_n c\in \SO_{2l}$ via the embedding mentioned before). Recall that
$$w^{l,n}=\left(\begin{matrix}
 & I_n & & & \\
 I_{l-n-1} & & & & \\
 & & I_2 & & \\
 & & & & I_{l-n-1} \\
 & & & I_n &
\end{matrix}
\right)\in\SO_{2l}.$$ Let 
$$
\tilde{w}_n=w^{l,n}\hat{w}_n(w^{l,n})\inv=\left(\begin{matrix}
& & & & I_n \\
& I_{l-n-1} & & & \\
& & X & & \\
& & & I_{l-n-1} & \\
I_n & & & & 
\end{matrix}\right),
$$
where 
$$
X=\left\{\begin{array}{cc}
  J_2  & \mathrm{if} \, n \, \mathrm{is} \,  \mathrm{odd,} \\
   I_2  & \mathrm{if} \, n \, \mathrm{is} \,  \mathrm{even.} \\
\end{array}\right. 
$$
 This Weyl element $\tilde{w}_n$  occurs in the computations of the zeta integrals of the twists by $\GL_n$ and is also used to partition the Bessel support. $\tilde{w}_n$ also shows up naturally in the computations of the zeta integrals after applying the intertwining operator.

Let $\alpha_i$ be the simple roots of $\SO_{2l}$ given by $\alpha_i (t)=\frac{t_i}{t_{i+1}}$ for $i\leq l-1$ and $\alpha_l(t)=t_{l-1}t_l$, where $t=\mathrm{diag}(t_1,\dots,t_l,t_l\inv,\dots,t_1\inv).$ Let $\Phi(\SO_{2l})$, resp. $\Phi^+(\SO_{2l})$, be the set of roots, resp. positive roots, of $\SO_{2l}.$ 
Then, for $n\leq l-2$, the action of $\tilde{w}_n$ on the simple roots is given by
\begin{align*}
    \tilde{w}_{n}\alpha_i &= \alpha_{n-i} \, \, \mathrm{for} \, \, 1\leq i \leq n-1,\\
    \tilde{w}_n\alpha_n (t) &= t_1\inv t_{n+1}\inv,\\
    \tilde{w}_n\alpha_i&=\alpha_i \, \, \mathrm{for} \, \, n+1\leq i \leq l-2,\\
    \tilde{w}_n \alpha_{l-1}&=\left\{\begin{array}{cc}
\alpha_l & \mathrm{if} \, n \, \mathrm{is} \,  \mathrm{odd,} \\
   \alpha_{l-1}  & \mathrm{if} \, n \, \mathrm{is} \,  \mathrm{even},
   \end{array}\right.\\
   \tilde{w}_n \alpha_{l}&=\left\{\begin{array}{cc}
\alpha_{l-1} & \mathrm{if} \, n \, \mathrm{is} \,  \mathrm{odd,} \\
   \alpha_{l}  & \mathrm{if} \, n \, \mathrm{is} \,  \mathrm{even.} 
\end{array}\right.
\end{align*}
Note that $\tilde{w}_{l-1}$ sends $\alpha_{l-1}$ and $\alpha_l$ to negative roots and the rest to simple roots. Specifically,
\begin{align*}
    \tilde{w}_{l-1}\alpha_i &= \alpha_{l-1-i} \, \, \mathrm{for} \, \, 1\leq i \leq l-2,\\
    \tilde{w}_{l-1} \alpha_{l-1}(t)&=\left\{\begin{array}{cc}
  t_1\inv t_l\inv & \mathrm{if} \, l-1 \, \mathrm{is} \,  \mathrm{odd,} \\
   t_1\inv t_l  & \mathrm{if} \, l-1 \, \mathrm{is} \,  \mathrm{even},
   \end{array}\right.\\
   \tilde{w}_{l-1} \alpha_{l}(t)&=\left\{\begin{array}{cc}
 t_1\inv t_l & \mathrm{if} \, l-1 \, \mathrm{is} \,  \mathrm{odd,} \\
    t_1\inv t_l\inv  & \mathrm{if} \, l-1 \, \mathrm{is} \,  \mathrm{even.} 
\end{array}\right.
\end{align*}

Next, we consider the case $n=l.$ Note that the construction of $\tilde{w}_n$ for $n<l$ arose from the embedding of the Weyl elements involved in the intertwining operator and the Weyl elements in the zeta integrals. Reversing this for $n=l,$ would require knowing what Weyl elements in $\SO_{2l}$ embed into certain subsets involving the Weyl element $w_l\in \SO_{2l+1}.$ We  answer this question later in Propositions \ref{l-1 twist embed}, \ref{l twist supp}, and \ref{l twist conj supp}. For now, we motivate our choices by using the following bijection between $\mathcal{P}(\Delta(\SO_{2l}))$ and the Bessel support.

Let $\theta_w=\{\alpha\in\Delta(\SO_{2l}) \, | \, w \alpha\in\Phi^+(\SO_{2l})\}.$ The assignment $w\mapsto \theta_w$ gives a bijection from $\mathrm{B}(\SO_{2l})$ to the power set of $\Delta(\SO_{2l})$, which we denote by $\mathcal{P}(\Delta(\SO_{2l}))$. Then, for $n<l-1,$ $\theta_{\tilde{w}_n}=\Delta(\SO_{2l})\setminus \{\alpha_n\}$. Also, we have $\theta_{\tilde{w}_{l-1}}=\Delta(\SO_{2l})\setminus \{\alpha_{l-1}, \alpha_l \}.$ In the cases of \cite{LZ}, the Weyl elements for which $\theta_w=\Delta(\SO_{2l})\setminus\{\alpha_i\}$ for some $i$ are used to partition the Bessel support. This suggests that we  need at least $2$ more Weyl elements which we define below. These  correspond to the sets $\Delta(\SO_{2l})\setminus \{\alpha_{l-1}\}$ and $\Delta(\SO_{2l})\setminus \{ \alpha_l \}.$

We extend the definition of $\tilde{w}_n$ to the case of $n=l$. We begin by defining $\tilde{w}'_l$ and  define $\tilde{w}_l$ as either this Weyl element or its conjugation by $c$ depending on the parity of $l$.
If $l$ is even, let
$$
\tilde{w}'_l=\left(\begin{matrix}
 & I_l \\
 I_l & \\
\end{matrix}\right),
$$
while if $l$ is odd, let 
$$
\tilde{w}'_l=\left(\begin{matrix}
 & & I_{l-1} & \\
 1 & & & \\
 & & & 1 \\
& I_{l-1} & & \\
\end{matrix}\right).
$$
We let $w_{\mathrm{long}}$ be the long Weyl element of $\SO_{2l}.$ Then, 

$$
w_{\mathrm{long}} =\left\{\begin{array}{cc}
  \left(\begin{matrix}
   &  & J_{l-1} \\
   & I_2 & \\
   J_{l-1} & & 
  \end{matrix}\right)  & \mathrm{if} \, l \, \mathrm{is} \,  \mathrm{odd,} \\
   J_{2l}  & \mathrm{if} \, l \, \mathrm{is} \,  \mathrm{even.} \\
\end{array}\right. 
$$
We have that $$w_{\mathrm{long}}\tilde{w^\prime_l}
=\left(\begin{matrix}
J_l & \\
& J_l
\end{matrix}\right)
$$
which is the long Weyl element of a standard Levi subgroup that is isomorphic to $\GL_l.$ This suggests that $\theta_{\tilde{w}'_l}$ should be $\Delta(\SO_{2l})\setminus \{\alpha_{l-1}\}$ or $\Delta(\SO_{2l})\setminus \{ \alpha_l \}.$ We verify this explicitly below. Also, its conjugation by $c$ is the remaining set since $c$ acts on the simple roots by swapping $\alpha_{l-1}$ and $\alpha_l$ and fixing the rest. Set $\tilde{w}'_l{}^c=c\tilde{w}'_lc.$

We have $\tilde{w}'_l t  (\tilde{w}'_l)\inv=\mathrm{diag}(t_l\inv,\dots,t_2\inv,t_1^{(-1)^{l+1}},t_1^{(-1)^{l}},t_2,\dots,t_l).$ Hence,
\begin{align*}
    \tilde{w}'_l \alpha_i&=\alpha_{l-i}\, \, \mathrm{for} \, \, 1\leq i \leq l-2,\\
    \tilde{w}'_l \alpha_{l-1}(t)&=\left\{\begin{array}{cc}
t_2\inv t_1\inv & \mathrm{if} \, l \, \mathrm{is} \,  \mathrm{odd,} \\
   \alpha_{1}(t)  & \mathrm{if} \, l \, \mathrm{is} \,  \mathrm{even,} \\
\end{array}\right.\\
\tilde{w}'_l \alpha_{l}(t)&=\left\{\begin{array}{cc}
\alpha_1 (t) & \mathrm{if} \, l \, \mathrm{is} \,  \mathrm{odd,} \\
   t_2\inv t_1\inv  & \mathrm{if} \, l \, \mathrm{is} \,  \mathrm{even.} \\
\end{array}\right.
\end{align*}
Thus, $\theta_{\tilde{w}'_l}=\Delta(\SO_{2l})\setminus\{\alpha_{l-1}\}$ if $l$ is odd and $\theta_{\tilde{w}'_l}=\Delta(\SO_{2l})\setminus\{\alpha_{l}\}$ if $l$ is even. Similarly, we have $\theta_{\tilde{w}'_l{}^c}=\Delta(\SO_{2l})\setminus\{\alpha_{l}\}$ if $l$ is odd and $\theta_{\tilde{w}'_l{}^c}=\Delta(\SO_{2l})\setminus\{\alpha_{l-1}\}$ if $l$ is even. 
Hence, we set 
$$
\tilde{w}_l =\left\{\begin{array}{cc}
 \tilde{w}'_l{}^c  & \mathrm{if} \, l \, \mathrm{is} \,  \mathrm{odd,} \\
   \tilde{w}'_l  & \mathrm{if} \, l \, \mathrm{is} \,  \mathrm{even}, \\
\end{array}\right. 
$$
 and $\tilde{w}_l^c=c\tilde{w}_l c$. Then, $\theta_{\tilde{w}_l}=\Delta(\SO_{2l})\setminus\{\alpha_{l}\}$ and $\theta_{\tilde{w}^c_l}=\Delta(\SO_{2l})\setminus\{\alpha_{l-1}\}.$
 
For $x\in\GL_n$, we set $t_n(x)=\mathrm{diag}(x,I_{2l-2n},x^*)\in\SO_{2l}.$ For $n\leq l-2$, let $\mathrm{B}_n(\SO_{2l})$ be the set of $w\in\mathrm{B}(\SO_{2l})$ such that there exists $w^\prime\in W(\GL_n)$ such that $w=t_n(w^\prime)\tilde{w}_n.$
We also let $\mathrm{B}_l(\SO_{2l})$ be the set of $w\in\mathrm{B}(\SO_{2l})$ such that there exists $w^\prime\in W(\GL_l)$ such that $w=t_l(w^\prime)\tilde{w}_l$ and $cwc\neq w.$ Also, let $\mathrm{B}^c_l(\SO_{2l})$ be the set of $cwc$ where $w\in\mathrm{B}_l(\SO_{2l}).$ We let $\mathrm{B}_{l-1}(\SO_{2l})$ be the set of $w\in\mathrm{B}(\SO_{2l})$ such that there exists $w^\prime\in W(\GL_l)$ such that $w=t_l(w^\prime)\tilde{w}_l$ and $cwc= w.$ By convention, we also define $\mathrm{B}_0(\SO_{2l})=\{I_{2l}\}.$ Note that $\theta_{I_{2l}}=\Delta(\SO_{2l}).$ From the definition, we have that $\mathrm{B}^c_l(\SO_{2l})$ is the set of $w^\prime\in W(\GL_l)$ such that $w=t_l^c(w^\prime)\tilde{w}_l^c$ and $cwc\neq w$ where $t_l^c(w^\prime)=ct_l(w^\prime)c.$ From Proposition \ref{l-1 twist} below, a similar definition holds for $\mathrm{B}_{l-1}(\SO_{2l}).$ That is, $\mathrm{B}_{l-1}(\SO_{2l})$ is the set of $w\in\mathrm{B}(\SO_{2l})$ such that there exists $w^\prime\in W(\GL_{l-1})$ such that $w=t_{l-1}(w^\prime)\tilde{w}_{l-1}.$

We remark on the significance of the condition $cwc\neq w$ in $\mathrm{B}_l(\SO_{2l})$. Let $\widetilde{\mathrm{B}_l}(\SO_{2l})$ be the set of $w\in\mathrm{B}(\SO_{2l})$ such that there exists $w^\prime\in W(\GL_l)$ such that $w=t_l(w^\prime)\tilde{w}_l$ and $\widetilde{\mathrm{B}_l^c}(\SO_{2l})$ be the set of $cwc$ such that $w\in \widetilde{\mathrm{B}_l}(\SO_{2l}).$ Then $\widetilde{\mathrm{B}_l}(\SO_{2l})\cap\widetilde{\mathrm{B}_l^c}(\SO_{2l})=\mathrm{B}_{l-1}(\SO_{2l})$ is a nonzero intersection and hence we would not be able to partition the Bessel support with these sets. Thus it is necessary to include the condition $cwc\neq w$ in the definitions of $\mathrm{B}_l(\SO_{2l})$ and $\mathrm{B}^c_l(\SO_{2l})$. This does not happen for the classical groups considered in \cite{LZ}. 

Later, we see that the twists by $\GL_n$ determine the normalized Bessel function on Bruhat cells associated to Weyl elements in $\mathrm{B}_n(\SO_{2l})$, for $n\leq l-2$, the twists by $\GL_{l-1}$ determine only a part of the Bruhat cells for those in $\mathrm{B}_{l-1}(\SO_{2l})$, and the twists by $\GL_l$ determine the rest of the Bruhat cells for $\mathrm{B}_{l-1}(\SO_{2l})$ and the cells of $\mathrm{B}_{l}(\SO_{2l})$ and $\mathrm{B}_{l}^c(\SO_{2l}).$ The goal for the rest of this section is to show that these sets partition the Bessel support. Before we show this, we present some computational results for $\mathrm{B}_{l-1}(\SO_{2l}).$
 
\begin{prop}\label{l-1 Besselpart}
Let $w\in\mathrm{B}_{l-1}(\SO_{2l})$ and $w^\prime\in W(\GL_l)$ such that $w=t_l(w^\prime)\tilde{w}_l$. Then 
$$
w^\prime=\left(\begin{matrix}
& w'' \\
1 & 
\end{matrix}\right),
$$
where $w''\in W(\GL_{l-1}).$ 

Conversely, if $w\in\mathrm{B}(\SO_{2l})$ is such that $w=t_l(w^\prime)\tilde{w}_l$ where
$$
w^\prime=\left(\begin{matrix}
& w'' \\
1 & 
\end{matrix}\right),
$$ for some $w''\in W(\GL_{l-1}),$ then $w\in\mathrm{B}_{l-1}(\SO_{2l})$.
\end{prop}

\begin{proof}
We only prove the first claim as the second claim is easy to verify directly.
Since $w\in\mathrm{B}_{l-1}(\SO_{2l})$, we have $cwc=w.$ Thus, $t_l(w^\prime)\tilde{w}_l=ct_l(w^\prime)c c\tilde{w}_l c$. First, assume that $l$ is odd, then $c\tilde{w}_l c= \tilde{w}'_l.$ Also $ct_l(w^\prime)=t_l(w^\prime)\tilde{w}_l \tilde{w}'_l c$ and 
$$\tilde{w}_l \tilde{w}'_l =\left(\begin{matrix}
& & & &  1 \\
& I_{l-2} & & & \\
& & J_2 & & \\
& & & I_{l-2} & \\
1 & & & &
\end{matrix}\right).$$
Let $w'=(w'_{i,j})_{i,j=1}^l.$ Then, 
$$t_l(w^\prime)\tilde{w}_l \tilde{w}'_l c=\left(\begin{matrix}
0 & (w'_{1,j})_{j=2}^l & 0 & w'_{1,1} \\
0 & (w'_{i,j})_{i,j=2}^l & 0 & (w'_{i,1})_{i=2}^l \\
* & 0 & * & 0 \\
* & 0 & * & 0
\end{matrix}\right),
$$
where the $*$'s represent coordinates in $(w')^*.$ On the other hand, $$ct_l(w^\prime)=
\left(\begin{matrix}
(w'_{i,j})_{i,j=1}^{l-1} & (w'_{i,l})_{i=1}^l & 0 & 0 \\
0 & 0 & * & * \\
w'_{l,1} & (w'_{l,j})_{j=2}^l & 0 & 0 \\
0 & 0 & * & *
\end{matrix}\right).
$$
Hence, we must have $w'_{i,1}=0$ for any $i=1,\dots,l-1$. Since $w'$ is a Weyl element of $\GL_l,$ $w'_{l,1}=1$ and hence the claim follows.

Next, we consider the case that $l$ is even. Again, we have $t_l(w^\prime)\tilde{w}_l=ct_l(w^\prime)c c\tilde{w}_l c$ and hence $t_l(w^\prime)\tilde{w}_l c\tilde{w}_l c c=ct_l(w^\prime)$. Again, $$\tilde{w}_l c\tilde{w}_l c =\left(\begin{matrix}
& & & &  1 \\
& I_{l-2} & & & \\
& & J_2 & & \\
& & & I_{l-2} & \\
1 & & & &
\end{matrix}\right).$$
The rest of this case follows exactly as above. This completes the proof of the proposition.
 \end{proof}

The following proposition relates $\mathrm{B}_{l-1}(\SO_{2l})$ with $\tilde{w}_{l-1}.$

\begin{prop}\label{l-1 twist}
Suppose that $w\in\mathrm{B}_{l-1}(\SO_{2l})$. Then there exists $\omega\in W(\GL_{l-1})$ such that $w=t_{l-1}(\omega)\tilde{w}_{l-1}.$ Conversely, if $w\in\mathrm{B}(\SO_{2l})$ is such that $w=t_{l-1}(\omega)\tilde{w}_{l-1}$ for some $\omega\in W(\GL_{l-1})$, then $w\in\mathrm{B}_{l-1}(\SO_{2l})$.
\end{prop}

\begin{proof}
By Proposition \ref{l-1 Besselpart}, $w=t_l(w^\prime)\tilde{w}_l$ where
$$
w^\prime=\left(\begin{matrix}
& w'' \\
1 & 
\end{matrix}\right),
$$
for some $w''\in W(\GL_{l-1}).$ We show that the claim follows with $\omega=w''.$ 

If $l$ is even, then $$w=t_l(w^\prime)\tilde{w}_l
=\left(\begin{matrix}
& & w'' \\
& J_2 & \\
(w'')^* & & 
\end{matrix}
\right).$$ It is immediate to check then $
t_{l-1}(w'')\tilde{w}_{l-1}=w.$ Hence $\omega=w''$ gives the claim.

If $l$ is odd, then $$w=t_l(w^\prime)\tilde{w}_l
=\left(\begin{matrix}
& & w'' \\
& I_2 & \\
(w'')^* & & 
\end{matrix}\right).
$$
Again, it is immediate to check then $
t_{l-1}(w'')\tilde{w}_{l-1}=w.$ Hence $\omega=w''$ gives the claim. This completes the proof of the proposition.
 \end{proof} 
 
For $0\leq n\leq l$, let $$P_n=\{ \theta\subseteq\Delta(\SO_{2l}) \, | \, w_\theta\in \mathrm{B}_n(\SO_{2l})\},$$ and $$P_l^c=\{ \theta\subseteq\Delta(\SO_{2l}) \, | \, w_\theta\in \mathrm{B}^c_l(\SO_{2l})\}.$$ The next proposition is crucial to show that the sets $\mathrm{B}_n(\SO_{2l})$ for $0\leq n\leq l$ and $\mathrm{B}_l^c(\SO_{2l})$ form a partition of the Bessel support.

\begin{prop}\label{Subsets}
For $n<l-1,$ $$P_n=\{\theta\in\Delta(\SO_{2l}) \, | \, \{\alpha_{n+1},\dots,\alpha_l\}\subseteq\theta\subseteq\Delta(\SO_{2l})\setminus\{\alpha_n\}\}.$$ We also have $P_0=\Delta(\SO_{2l}),$
$$
P_l=\{\theta\in\Delta(\SO_{2l}) \, | \, \{\alpha_{l-1}\}\subseteq\theta\subseteq\Delta(\SO_{2l})\setminus\{\alpha_{l}\}\},
$$
$$P_l^c=\{\theta\in\Delta(\SO_{2l}) \, | \, \{\alpha_{l}\}\subseteq\theta\subseteq\Delta(\SO_{2l})\setminus\{\alpha_{l-1}\}\},$$ and $$P_{l-1}=\{\theta\in\Delta(\SO_{2l}) \, | \, \theta\subseteq\Delta(\SO_{2l})\setminus\{\alpha_{l-1},\alpha_l\}\}.$$ 
\end{prop}

\begin{proof}
First, assume that $n\leq l-2$ and suppose $\theta\in P_n.$ Then, $w_\theta\in \mathrm{B}_n(\SO_{2l}).$ Thus, there exists $w'\in W(\GL_n)$ such that $w=t_n(w')\tilde{w}_n.$ $t_n(w')$ acts on $t=\mathrm{diag}(t_1,\dots,t_l,t_l\inv,\dots,t_1\inv)$ by permuting the $t_i$'s for $1\leq i \leq n.$ By direct computation,
$\{\alpha_{n+1},\dots,\alpha_l\}\subseteq\theta\subseteq\Delta(\SO_{2l})\setminus\{\alpha_n\}$.

Suppose that $n\leq l-2$ and $\{\alpha_{n+1},\dots,\alpha_l\}\subseteq\theta\subseteq\Delta(\SO_{2l})\setminus\{\alpha_n\}$. First we  show that $w_\theta$ takes the set $\{\alpha_{l-1},\alpha_l\}$ to itself. Suppose that $w_\theta \alpha_l=\alpha_j$ for some $j\leq l-2$ for contradiction. Let $t=\mathrm{diag}(t_1,\dots,t_l,t_l\inv,\dots,t_1\inv)$ and $s=w_\theta t w_\theta\inv=\mathrm{diag}(s_1,\dots,s_l,s_l\inv,\dots,s_1\inv).$ Then, $w_\theta \alpha_l (t)=t_jt_{j+1}\inv.$ So, $w_\theta \alpha_{l-1}(t)= t_j t_{j+1}$. That is, $w_\theta \alpha_{l-1}$ is a positive, nonsimple root. But $w_\theta$ should be in the Bessel support and hence we have a contradiction. Therefore, $w_\theta$ takes the set $\{\alpha_{l-1},\alpha_l\}$ to itself. So, $s_{l-1}=t_{l-1}$ and $s_l=t_l$ or $s_l=t_l\inv.$ 

Next, we show that $w_\theta \alpha_i=\alpha_i$ for $n+1 \leq i\leq l-2.$ We do this by considering $\alpha_{l-2},$ then $\alpha_{l-3},$ and so on until we reach $\alpha_{n+1}.$
We have $w_\theta \alpha_{l-2}(t)=t_j^\varepsilon t_{l-1}\inv$ for some $j\leq l-2$ and $\varepsilon=\pm 1.$ If $\varepsilon=-1,$ this is a negative root, but it should be simple by assumption. Thus $\varepsilon=1$ and $w_\theta \alpha_{l-2}$ is a positive root. In this case, $w_\theta \alpha_{l-2}$ is simple if and only if $j=l-2.$ Hence, $s_{l-2}=t_{l-2}.$ 
We repeat this argument for $\alpha_{l-3}$ and then $\alpha_{l-4}$ and so on until we obtain $\alpha_{n+1}.$ In summary, we have $w\alpha_i=\alpha_i$ for $n+1 \leq i\leq l-2.$ Hence, $s_i=t_i$ for $n+1 \leq i\leq l-1$ and $s_l\in\{t_l,t_l\inv\}.$

Then, $w_\theta \alpha_n (t)=t_j^\varepsilon t_{n+1}\inv$ for some $j\leq n$ and $\varepsilon=\pm 1.$ Since  $w_\theta \alpha_n$ is negative, $\varepsilon=-1.$ Finally, since $w_\theta \alpha_i$ is either simple or negative for $i\leq n-1$, it follows that $s_i \in \{t_1\inv,\dots,t_n\inv\}$ for any $i\leq n.$ Let $w'\in W(\GL_n)$ be the Weyl element which takes $\mathrm{diag}(t_1,\dots,t_n)$ to $\mathrm{diag}(s_1\inv,\dots,s_n\inv).$ Then, $w_\theta
=t_n(w')\tilde{w}_n
$
for some $w'\in W(\GL_n)$ (note that the $X$ in the definition of $\tilde{w}_n$ comes from needing $\mathrm{det}(w_\theta)=1$ and this also determines the action of $w_\theta$ on $\alpha_{l-1}$ and $\alpha_l$). Therefore, for $n<l-1,$ we have 
$$P_n=\{\theta\in\Delta(\SO_{2l}) \, | \, \{\alpha_{n+1},\dots,\alpha_l\}\subseteq\theta\subseteq\Delta(\SO_{2l})\setminus\{\alpha_n\}\}.$$

Next, we examine $P_{l-1}.$ Suppose $\theta\in P_{l-1}.$ Then, $w_\theta\in \mathrm{B}_{l-1}(\SO_{2l}).$ Thus, there exists $w'\in W(\GL_l)$ such that $w_\theta=t_l(w')\tilde{w}_l.$ Also, $cw_\theta c=w_\theta.$
By Proposition \ref{l-1 Besselpart}, $$
w^\prime=\left(\begin{matrix}
& w'' \\
1 & 
\end{matrix}\right),
$$
for some $w''\in W(\GL_{l-1}).$ Since $\tilde{w}_l t \tilde{w}_l\inv=\mathrm{diag}(t_l^{\pm 1},\dots,t_1\inv,t_1,\dots, t_l^{\mp 1}),$ it follows that the $(l,l)$-entry of $w_\theta t w_\theta\inv$ is $t_l^{\pm 1}$ (the sign depends on the parity of $l$). Hence $w_\theta \alpha_l(t)=t_j\inv t_l^{\pm 1}$ for some $j\leq l-1$ and hence $w_\theta \alpha_l$ is a negative root. So, $\alpha_l \notin \theta.$ $w_\theta \alpha_{l-1}(t)= t_j\inv t_l^{\mp 1}$ is also a negative root. So $\alpha_{l-1}\notin \theta$ and hence $\theta\subseteq \Delta(\SO_{2l})\setminus\{\alpha_{l-1},\alpha_l\}.$

Now we suppose that $\theta\subseteq \Delta(\SO_{2l})\setminus\{\alpha_{l-1},\alpha_l\}.$ Then $w_\theta \alpha_{l-1}$ and $w_\theta \alpha_{l-1}$ are negative. Let $s=w_\theta t w_\theta\inv=\mathrm{diag}(s_1,\dots,s_l,s_l\inv,\dots,s_1\inv).$ Then $w_\theta \alpha_{l-1}(t)=s_{l-1}s_l\inv$ and $w_\theta \alpha_l(t)=s_{l-1}s_l.$ We  show that $s_{l-1}=t_j\inv$ for some $j$ by contradiction. Suppose $s_{l-1}=t_j$ If $s_l=t_k$, then  $w_\theta \alpha_l(t)=t_j t_k$ is a positive root which is a contradiction. But, if $s_l=t_k\inv$, then  $w_\theta \alpha_{l-1}(t)=t_j t_k$ is positive and hence we have another contradiction. Therefore, we must have $s_{l-1}=t_j\inv.$

Our next step is to show that $s_l=t_l$ or $s_l=t_l\inv$ depending on the parity of $l$. Since $s_{l-1}=t_j\inv,$ $w_\theta \alpha_{l-2}(t)=s_{l-2}t_j.$ Since $w_\theta\in\mathrm{B}(\SO_{2l})$, this root must be simple or negative. It is simple if $s_{l-2}=t_{j+1}\inv.$ It is negative if $s_{l-2}=t_r\inv$ for some $r<j.$ In either case we have that $s_{l-2}=t_r\inv$ for some $r$. By continuing this process for all the roots, we find that for each $n<l$ there exists a unique $j_n$ such that $s_n=t_{j_n}\inv$ Suppose that $s_l\neq t_l$ and $s_l\neq t_l\inv$ for contradiction. Then, there exists $n$ such that $s_n=t_l\inv.$ Then $w_\theta \alpha_n(t)=t_l\inv t_{j_{n+1}}.$ This root is positive for any index $j_{n+1}.$ It is simple when $j_{n+1}=l-1.$ Thus we must have $s_{n+1}=t_{l-1}\inv.$ Next, $w_\theta \alpha_{n+1}(t)=t_{l-1}\inv t_{j_{n+2}}.$ This root is positive for any index $j_{n+2}<l-1.$ It is simple when $j_{n+1}=l-2.$ Thus we must have $s_{n+2}=t_{l-2}\inv.$ Continuing in this fashion we find that $s_{i}=t_{l+n-i}\inv$ for $n\leq i \leq l-2.$ Thus $w_\theta \alpha_{l-1}(t)= t_{l+n-(l-2)}\inv t_j$ where $j\leq l+n-(l-2).$ This is always a postive root. However, by assumption, it is negative. Hence we have a contradiction. Therefore, we must have $s_l=t_l$ or $s_l=t_l\inv.$
Thus, $w_\theta t w_\theta\inv = \mathrm{diag}(t_{j_1}\inv,\dots,t_{j_{l-1}}\inv, t_l^\epsilon, t_l^{-\epsilon}, t_{j_{l-1}}, \dots, t_{j_{1}})$ for some $\epsilon\in\{\pm 1\}.$ Let $w''\in W(\GL_{l-1})$ be such that $w''\mathrm{diag}(t_1,\dots,t_{l-1})(w'')\inv=\mathrm{diag}(t_{j_{l-1}}, \dots, t_{j_{1}}).$ Then, $t_{l-1}(w'')\tilde{w}_{l-1}=w_\theta$ (note that the parity of $l$ determines $\epsilon$ since we need the determinant to be 1) and hence $w_\theta\in\mathrm{B}(\SO_{2l}).$ Therefore, we have 
$$P_{l-1}=\{\theta\in\Delta(\SO_{2l}) \, | \, \theta\subseteq\Delta(\SO_{2l})\setminus\{\alpha_{l-1},\alpha_l\}\}.$$

Next, we show the claim for $P_l.$ Suppose that $w_\theta=t_l(w')\tilde{w}_l$ for some $w'\in W(\GL_l)$ with $cw_\theta c\neq w_\theta.$ For $l$ odd, we have 
$$\tilde{w}_l t \tilde{w}_l\inv = \mathrm{diag}(t_l,t_{l-1}\inv,\dots, t_2\inv, t_1\inv, t_1, t_2,\dots, t_{l-1}, t_l\inv),$$
and for $l$ even
$$\tilde{w}_l t \tilde{w}_l\inv = \mathrm{diag}(t_l\inv,t_{l-1}\inv,\dots, t_2\inv, t_1\inv, t_1, t_2,\dots, t_{l-1}, t_l).$$
Let $w'=(w'_{i,j})_{i,j=1}^l.$
Since $cw_\theta c \neq w_\theta,$ by Proposition \ref{l-1 Besselpart}, 
$w'_{l,1}=0.$
Let 
$$w_\theta t w_\theta\inv=\mathrm{diag}(s_1,\dots,s_l,s_l\inv,\dots,s_1\inv).$$
For $l$ odd, $s_i\in\{t_1\inv,\dots,t_{l-1}\inv,t_l\}$ and for $l$ even $s_i\in\{t_1\inv,\dots,t_{l-1}\inv,t_l\inv\}.$ Since $w'_{l,1}=0,$ $s_l\neq t_l^{\pm 1}.$ Thus, $w_\theta \alpha_{l-1}=s_{l-1}s_l\inv=s_{l-1}t_r$  and $w_\theta \alpha_{l-1}=s_{l-1}s_l=s_{l-1}t_r\inv$ for some $r<l$. If $s_{l-1}=t_l,$ then $w_\theta \alpha_{l-1} $ is positive and we must have $r=l-1$ so that it is simple. Then, $w_\theta \alpha_l=t_l t_{l-1}\inv$ is negative and hence we have $\{\alpha_{l-1}\}\subseteq\theta\subseteq\Delta(\SO_{2l})\setminus\{\alpha_{l}\}.$ If $s_{l-1}=t_l\inv$, then $w_\theta \alpha_{l-1}$ is again positive since $r<l.$ We must have $r=l-1$ to ensure it is simple. But, then $w_\theta \alpha_{l-2}(t)=t_j t_l$ for some $j<l-1$. This root is positive and never simple. Hence we have a contradiction. Thus $s_{l-1}\neq t_l\inv.$ Thus we have shown the claim for $s_{l-1}=t_l$ (and $t_l\inv$ for which we had a contradiction) so far. Suppose $s_{l-1}=t_j\inv$ for some $j<l.$ Then $w_\theta \alpha_{l-1}(t)=t_j\inv t_r$ and $w_\theta \alpha_{l}(t)=t_j\inv t_r\inv.$ So $w_\theta \alpha_{l}$ is always negative. $w_\theta \alpha_{l-1}$ is negative if $j<r$ and simple if $j=r+1.$ We  show that we must have $j=r+1.$ Suppose $j<r$ for contradiction. Then there exists $n$ such that $s_n=t_l$ or $t_l\inv.$ Suppose $s_n=t_l$. Then $w_\theta \alpha_{n}(t)=s_n s_{n+1}\inv=t_l t_{j_{n+1}}.$ This root is positive and only simple if $j_{n+1}=l-1.$ Then $w_\theta \alpha_{n+1}(t)=t_{l-1}\inv t_{j_{n+2}}$ for some $j_{n+2}<l-1.$ Again this root is always positive and is simple only if $j_{n+2}=l-2.$ Continuing this, we find $s_i=t_{l+n-i}\inv$ for $n+1\leq i \leq l.$ Thus we have $k=n$ and $w_\theta \alpha_{l-1}(t)=t_{n+1}\inv t_{n}$ is simple and $w_\theta \alpha_{l}(t)=t_{n+1}\inv t_{n}\inv$ is negative. Therefore, $\{\alpha_{l-1}\}\subseteq\theta\subseteq\Delta(\SO_{2l})\setminus\{\alpha_{l}\}.$ We address the final case. Suppose that $s_n=t_l\inv.$ Then $w_\theta \alpha_n(t)=t_l\inv t_{j_{n+1}}$ with $j_{n+1}<l.$ This root is always positive. It is simple only if $j_{n+1}=l-1.$ The rest of this case follows exactly as in the previous case. Again we find
$\{\alpha_{l-1}\}\subseteq\theta\subseteq\Delta(\SO_{2l})\setminus\{\alpha_{l}\}$ and $\theta\in P_l.$

Next, suppose that $\{\alpha_{l-1}\}\subseteq\theta\subseteq\Delta(\SO_{2l})\setminus\{\alpha_{l}\}.$ Let 
$$w_\theta t w_\theta\inv=\mathrm{diag}(s_1,\dots,s_l,s_l\inv,\dots,s_1\inv).$$
Then, 
$w_\theta \alpha_{l-1}(t)=s_{l-1}s_l\inv$ is simple and $w_\theta \alpha_{l-1}(t)=s_{l-1}s_l$ is negative. Suppose first that the simple root $w_\theta \alpha_{l-1}(t)$ is 
$s_{l-1}s_l\inv=t_r t_{r+1}\inv$ for some $r<l$. If $s_l=t_{r+1},$ then $w_\theta \alpha_{l}(t)=t_r t_{r+1}$ is positive which is a contradiction. Therefore, we must have $s_{l-1}=t_{r+1}\inv$ and $s_l=t_r\inv.$
Hence, $w_\theta \alpha_{l-2}(t)=s_{l-2}s_{l-1}\inv=t_{j_{l-2}}^{\epsilon}t_{r+1}$ where $\epsilon\in\{\pm 1\}.$ This is simple if $\epsilon=-1$ and $j_{l-2}=r+2.$ It is negative if $\epsilon=-1$ and $j_{l-2}<r+1.$ In either case, we have $s_{l-2}=t_{j_{l-2}}\inv.$ We can continue this argument until we arrive at $s_{n+1}=t_{l-1}\inv.$ That is, $s_i=t_{j_{i}}\inv$ for $n+1 \leq i\leq l.$ The next root is $w_\theta \alpha_{n}(t)=s_{n}s_{n+1}\inv=t_{j_{n}}^{\epsilon}t_{l-1}$ where $\epsilon\in\{\pm 1\}.$ This is simple for two cases: when $j_n=l$ and $\epsilon=\pm 1.$ $\epsilon$ is determined by the parity of $l$. Specifically, if $\epsilon=1$, then $l$ must be odd and if $\epsilon=-1,$ then $l$ must be even. Suppose that $s_n=t_l.$ Then $w_\theta \alpha_{n-1}(t)=s_{n-1}s_{n}\inv=t_{j_{n-1}}^{\epsilon}t_{l}\inv$ where $j_{n-1}< r$ and $\epsilon\in\{\pm 1\}.$ Since $j_{n-1}< r< l-1$, this root can never be simple. Thus it must be negative and so $\epsilon=-1.$ That is, $s_{n-1}=t_{j_{n-1}}\inv.$ Continuing in the above manner, we again find that $s_i=t_{j_i}\inv$ for $i<n.$ Let $w''\in W(\GL_l)$ be such that $w'' \mathrm{diag}(t_l,t_{l-1}\inv,\dots,t_1\inv) (w'')\inv=\mathrm{diag}(s_1,\dots,s_l).$ Then, $w_\theta=t_l(w'')\tilde{w}_l\in \mathrm{B}_l(\SO_{2l}).$
Next, we must address the case  $s_n=t_l\inv.$ Then $w_\theta \alpha_{n-1}(t)=s_{n-1}s_{n}\inv=t_{j_{n-1}}^{\epsilon}t_{l}$ where $j_{n-1}< r$ and $\epsilon\in\{\pm 1\}.$ Since $j_{n-1}< r< l-1$, this root can never be simple. Thus it must be negative and so $\epsilon=-1.$ That is, $s_{n-1}=t_{j_{n-1}}\inv.$ Continuing in the above manner, we again find that $s_i=t_{j_i}\inv$ for $i<n.$ Let $w''\in W(\GL_l)$ be such that $w'' \mathrm{diag}(t_l\inv,t_{l-1}\inv,\dots,t_1\inv) (w'')\inv=\mathrm{diag}(s_1,\dots,s_l).$ Then, $w_\theta=t_l(w'')\tilde{w}_l\in \mathrm{B}_l(\SO_{2l}).$ Thus, if $w_\theta \alpha_{l-1}(t)=t_r t_{r+1}\inv$ for some $r<l$, then 
$$
P_l=\{\theta\in\Delta(\SO_{2l}) \, | \, \{\alpha_{l-1}\}\subseteq\theta\subseteq\Delta(\SO_{2l})\setminus\{\alpha_{l}\}\},
$$
Suppose that $w_\theta \alpha_{l-1}(t)$ is the remaining simple root
$s_{l-1}s_l\inv=t_{l-1} t_{l}$. If $s_{l-1}=t_{l-1}$ then $w_\theta\alpha_l$ is simple which is a contradiction. Thus, we must have $s_{l-1}=t_l$ and $s_l=t_{l-1}\inv.$

Thus, $w_\theta\alpha_{l-2}(t)=s_{l-2}t_l\inv.$ This root cannot be simple and hence we must have $s_{l-2}=t_{j_{l-2}}\inv.$ Next, $w_\theta\alpha_{l-3}(t)=s_{l-3}t_{j_{l-2}}=t_{j_{l-3}}^\epsilon t_{j_{l-2}}.$ This root is always positive and not simple if $\epsilon=1.$ Thus we must have $s_{l-3}=t_{j_{l-3}}\inv.$ Continuing in this manner for all the remaining roots shows that for any $i\leq l-2$ we have $s_i\in\{t_1\inv,\cdots,t_{l-2}\inv\}.$ 
Thus, $$w_\theta=\left(\begin{matrix}
 &  & w_1 \\
 & w_2 &  \\
w_1' & & \\
\end{matrix}\right),$$ 
where $w_1 \mathrm{diag}(t_{l-2}\inv,\dots,t_{1}\inv) w_1\inv=\mathrm{diag}(s_1,\dots,s_{l-2})$, and
$$
w_2=\left(\begin{matrix}
0 & 1 & 0 & 0 \\
0 & 0 & 0 & 1 \\
1 & 0 & 0 & 0 \\
0 & 0 & 1 & 0
\end{matrix}\right).
$$
Note that $\mathrm{det}(w_2)=-1$ and hence this is only possible if $l$ is odd (for $l$ even, $\mathrm{det}w_\theta=-1$).
Let $w''\in W(\GL_l)$ be such that $w'' \mathrm{diag}(t_l,t_{l-1}\inv,\dots,t_1\inv) (w'')\inv=\mathrm{diag}(s_1,\dots,s_l).$ Then, we have $w_\theta=t_l(w'')\tilde{w}_l\in \mathrm{B}_l(\SO_{2l}).$ Thus, in both the cases of $w_\theta \alpha_{l-1}(t)$, we have 
$$
P_l=\{\theta\in\Delta(\SO_{2l}) \, | \, \{\alpha_{l-1}\}\subseteq\theta\subseteq\Delta(\SO_{2l})\setminus\{\alpha_{l}\}\}.
$$

Finally, conjugation by $c$ fixes the roots $\alpha_i$ for $i\leq l-2$ and maps $\alpha_{l-1}\mapsto \alpha_l$ and $\alpha_l\mapsto \alpha_{l-1}.$ Using the result for $P_l$, it is straightforward to check that 
$$P_l^c=\{\theta\in\Delta(\SO_{2l}) \, | \, \{\alpha_{l}\}\subseteq\theta\subseteq\Delta(\SO_{2l})\setminus\{\alpha_{l-1}\}\}.$$ This finishes the proof of the proposition.
 \end{proof}

Next, we interpret the previous proposition in terms of the Bessel support.

\begin{prop}\label{BesselPartition}
The sets $\mathrm{B}_n(\SO_{2l})$ for $n=0,1,\dots,l$ and $\mathrm{B}^c_l(\SO_{2l})$ form a partition of $\mathrm{B}(\SO_{2l}).$
\end{prop}

\begin{proof}
We have that $$
\bigsqcup_{n=0}^{l-2} P_n \sqcup P_l \sqcup P_l^c \sqcup P_{l-1} = \mathcal{P}(\Delta(\SO_{2l})),
$$
the power set of $\Delta(\SO_{2l})$. 
The claim then follows from Proposition \ref{Subsets}.
 \end{proof}

We need the following proposition to show that the zeta integrals are nonzero in later sections.

\begin{prop}\label{uppertriangular}
If $w\in W(\GL_l)$ and $w\neq I_l$, then $t_l(w)\notin \mathrm{B}(\SO_{2l}).$ In particular, if $a\in\GL_l$ and $a$ is not upper triangular, then $\B_{\pi,\psi}(t_l(a))=0.$
\end{prop}

\begin{proof}
We  use the partition of the Bessel support in Proposition \ref{BesselPartition}.
First, suppose that $t_l(w)=t_n(w')\tilde{w}_n$ for some $n$ with $1\leq n\leq l$ and $w'\in W(\GL_n).$ Then, $\tilde{w}_n\in W(\GL_l).$ However, this is not the case for any $n$. Next, suppose $t_l(w)=ct_l(w')\tilde{w}_l c$ (this is the case that $t_l(w)\in \mathrm{B}_l^c(\SO_{2l})).$ Thus, $ct_l(w)c=t_l(w')\tilde{w}_l$. Let $w=(w_{i,j})_{i,j=1}^l.$ Then, $$
ct_l(w)c=\left(\begin{matrix}
(w_{i,j})_{i,j=1}^{l-1} & 0 & (w_{i,l})_{i=1}^{l-1} & 0 \\
0 & * & 0 & * \\
(w_{l,j})_{j=2}^{l-1} & 0 & w_{l,l} & 0 \\
0 & * & 0 & *
\end{matrix}\right).
$$

If $l$ is even then, $$
t_l(w')\tilde{w}_l=\left(\begin{matrix}
& w' \\
(w')^* &
\end{matrix}\right).
$$
Hence, $w_{l,l}=0$ and $w_{i,j}=0$ for any $1\leq i,j\leq l-1.$ Thus, $$
w=\left(\begin{matrix}
0_{(l-1)\times(l-1)}& * \\
* & 0
\end{matrix}\right).
$$
But, $w\in W(\GL_l)$ and hence this is not possible. 

If $l$ is odd, we have
$$
t_l(w')\tilde{w}_l=\left(\begin{matrix}
0_{l\times(l-1)} & (w'_{i,1})_{i=1}^{l} & 0_{l\times 1} & (w'_{i,j})_{i=1,j=2}^{l} \\
* & 0 & * & 0
\end{matrix}\right).
$$
Again, $w_{i,j}=0$ for any $1\leq i,j\leq l-1$, but since $w\in W(\GL_l)$, this is not possible. Therefore, $w\notin \mathrm{B}_l^c(\SO_{2l})$ and $w\notin \mathrm{B}_n(\SO_{2l})$ for any $n$. By Proposition \ref{BesselPartition}, $w\notin \mathrm{B}(\SO_{2l}).$

For the second part of the claim, suppose $a\in\GL_l$ is not upper triangular, Then, by the Bruhat decomposition of $\GL_l$, there exists $w\in W(\GL_l)$ such that $w\neq I_l$ and $a=u_1t w u_2$ where $t\in T_{\GL_l}$ and $u_1, u_2\in U_{\GL_l}.$ We obtain $\B_{\pi,\psi}(t_l(a))=\B_{\pi,\psi}(t_l(t)t_l(w))\psi(t_l(u_1))\psi(t_l(u_2)).$ However, $t_l(w)\notin \mathrm{B}(\SO_{2l})$ and hence $\B_{\pi,\psi}(t_l(t)t_l(w))=0.$ Thus, we have proven the proposition.
 \end{proof}

\section{Twists by $\GL_n$ for $n\leq l-2$}
In this section, we calculate the zeta integrals for twists up to $\GL_{l-2}.$
Let $v\in\tau$ be a fixed vector and define $\xi_v\in I(\tau)$ by supp$(\xi_v)=Q_n$ and
$$
\xi_v(l_n(a)u)=\tau(a)v, \forall a\in\GL_n, u\in V_n,
$$
where $l_n(a)=\mathrm{diag}(a,1,a^*)$ and $Q_n=L_n V_n$ is the standard Siegel parabolic of $\SO_{2n+1}.$
Let $f_v=f_{\xi_v}\in I(\tau,\psi^{-1}).$ That is, $f_v(g,a)=\Lambda_\tau(\tau(a)\xi_v(g))$, where $\Lambda_\tau\in \mathrm{Hom}_{U_{\GL_{n}}}(\tau,\psi^{-1}),$ $a\in\GL_n,$ and $g\in\SO_{2n+1}.$ We also fix the Whittaker function $W_v(a)=\Lambda_\tau(\tau(a)v).$ We show that the zeta integrals in this case are  nonzero.

\begin{prop}\label{GLnnonzero}
$\Psi(\B_{\pi,\psi},f_v)= W_v(I_n)$. In particular, we may choose $v\in\tau$ such that $\Psi(\B_{\pi,\psi},f_v)\neq 0.$
\end{prop}

\begin{proof}
By definition, $$\Psi(\B_{\pi,\psi},f_v)
=\sum_{g\in U_{\SO_{2n+1}}\setminus\SO_{2n+1}} \left(\sum_{r\in R^{l,n}} \B_{\pi,\psi}(r w^{l,n} g (w^{l,n})\inv)\right)   f_v(g, I_{n}).
$$
The support of $f_v(g, I_n)$ is contained in $V_n L_n$. Since $V_n\subseteq U_{\SO_{2n+1}}$, we have 
$$\Psi(\B_{\pi,\psi},f_v)
=\sum_{a\in U_{GL_n}\setminus\GL_n} \left(\sum_{r\in R^{l,n}} \B_{\pi,\psi}(r w^{l,n} l_n(a) (w^{l,n})\inv)\right)   W_v(a).
$$
Embedding of $\SO_{2n+1}$ into $\SO_{2l}$ takes $l_n(a)$ to $q_n(a)=\mathrm{diag}(I_{l-n-1},a,I_2,a^*,I_{l-n-1}).$ We have $w^{l,n} q_n(a) (w^{l,n})\inv=\mathrm{diag}(a, I_{2l-2n}, a^*).$ Let $$r_x=\left(\begin{matrix}
I_n & & & & \\
x & I_{l-n-1} & & & \\
& & I_2 & & \\
& & & I_{l-n-1} & \\
& & & x^\prime & I_n
\end{matrix}\right)\in R^{l,n}.$$ Then, $$
r_x w^{l,n} q_n(a) (w^{l,n})\inv
=\left(\begin{matrix}
a & & & & \\
xa & I_{l-n-1} & & & \\
& & I_2 & & \\
& & & I_{l-n-1} & \\
& & & x^\prime a^* & a^*
\end{matrix}\right).
$$
By Proposition \ref{uppertriangular}, $\B_{\pi,\psi}(r_x w^{l,n} q_n(a) (w^{l,n})\inv)=0$ unless $a$ is upper triangular and $x=0.$ Hence, 
$$\Psi(\B_{\pi,\psi},f_v)
=\sum_{a\in T_{\GL_n}} \B_{\pi,\psi}(t_n(a))   W_v(a).
$$
Finally, by Lemma \ref{center}, $\Psi(\B_{\pi,\psi},f_v)
=W_v(I_n).$ If we choose $v$ to be a Whittaker vector for $\tau$, then $W_v(I_n)\neq 0.$ This completes the proof of the proposition.
 \end{proof}

Let $\tilde{f}_v=M(\tau,\psi\inv)f_v$ and $W_v^*(a)=\Lambda_\tau(\tau(d_n a^*)v).$ The following lemma is the analogue of \cite[Lemma 3.6]{LZ} and holds for any $n$. The proof follows similarly.

\begin{lemma}\label{intertwine}
\begin{enumerate}
    \item If $\tilde{f}_v(g, I_n)\neq 0,$ then $g\in Q_n w_n Q_n = Q_n w_n V_n.$
    \item If $x\in V_n$, then $\tilde{f}_v(w_n x, I_n)=W^*_v(I_n).$
    \item For $a\in\GL_n$ and $x\in V_n,$ $\tilde{f}_v(l_n(a)w_n x, I_n)=W_v^*(a).$
\end{enumerate}
\end{lemma}

In \cite{LZ}, each twist utilizes \cite[Lemma 3.1]{Ni14}. 
For twists up to $\GL_{l-2},$ \cite[Lemma 3.1]{Ni14} would be sufficient. However, for twists by $\GL_{l-1}$ and $\GL_l$, it is expedient to have the below lemma. Note that \cite[Lemma 3.1]{Ni14} is precisely the case when $X=\GL_n$ in the below lemma.

\begin{lemma}\label{Nien}
Suppose $X$ is a subset of $\GL_n$ for which $ux\in X$ for any $u\in U_{\GL_n}$ and $x\in X.$ Let $H$ be a function on $X$ such that $H(ug)=\psi(u)H(g)$ for any $u\in U_{\GL_n}$ and $g\in X.$
Suppose that 
$$\sum_{x\in X} H(x)W_v(x)=0
$$
for any $v\in\tau$ and with $\tau$ running through all generic irreducible representations of $\GL_n.$ Then, $H(x)=0$ for any $x\in X.$
\end{lemma}

\begin{proof}
Let $H'$ be the function on $\GL_n$ defined by $H'(g)=H(g)$ if $g\in X$ and $H'(g)=0$ if $g\notin X.$ Then, $H'(ug)=\psi(u)H'(g)$ for any $u\in U_{\GL_n}$ and $g\in\GL_n.$ By \cite[Lemma 3.1]{Ni14}, $H'\equiv 0$ on $\GL_n$ and hence the lemma follows.
 \end{proof}

The following theorem shows that, for $n\leq l-2,$ the twists by $\GL_n$ determine that the Bessel functions are equal on $\mathrm{B}_n(\SO_{2l}).$ Note that for $\tilde{w}_n$ fixes $\alpha_{l-1}$ and $\alpha_l$ when $n$ is even and switches them when $n$ is odd. Recall that $t_n'=\tilde{t}=\mathrm{diag}(I_{l-1},\frac{-1}{2},-2,I_{l-1})$ if $n$ is odd and $t_n'=I_{2l}$ if $n$ is even. We explain the inclusion of the term $t_n'$ in the below theorem. Let $x\in\mathbb{F}_q$ and $\mathrm{\bold{x}}_{\alpha_l}(x)$ be an element of the root space of $\alpha_l.$ By Proposition \ref{Besselprop}, we have 
$$
\B_{\pi,\psi}(t_n(a) \tilde{w}_n \mathrm{\bold{x}}_{\alpha_l}(x)) = \B_{\pi,\psi}(t_n(a) \tilde{w}_n)(-\psi(x)/2).
$$
When $n$ is odd,
$$
\B_{\pi,\psi}(t_n(a) \tilde{t} \tilde{w}_n \mathrm{\bold{x}}_{\alpha_l}(x)) = \B_{\pi,\psi}(\mathrm{\bold{x}}_{\alpha_{l-1}}(x) t_n(a)\tilde{t}\tilde{w}_n)=\B_{\pi,\psi}( t_n(a)\tilde{t}\tilde{w}_n)( \psi(x)/4).
$$
Since $\psi$ is nontrivial, it follows that $\B_{\pi,\psi}(t_n(a) \tilde{w}_n)=0.$ However, if we include $t_n'=\tilde{t},$ we have
$$
\B_{\pi,\psi}(t_n(a) \tilde{t} \tilde{w}_n \mathrm{\bold{x}}_{\alpha_l}(x)) = \B_{\pi,\psi}(t_n(a)\tilde{t} \tilde{w}_n)(-\psi(x)/2)
$$
and
$$
\B_{\pi,\psi}(t_n(a) \tilde{t} \tilde{w}_n \mathrm{\bold{x}}_{\alpha_l}(x)) = \B_{\pi,\psi}(\mathrm{\bold{x}}_{\alpha_{l-1}}(-2x) t_n(a) \tilde{t} \tilde{w}_n)=\B_{\pi,\psi}(t_n(a) \tilde{t} \tilde{w}_n)( -\psi(x)/2).
$$
Thus, $\B_{\pi,\psi}(t_n(a) \tilde{t} \tilde{w}_n)$ may be nonzero.

When $n$ is even, we have
$$
\B_{\pi,\psi}(t_n(a) \tilde{w}_n \mathrm{\bold{x}}_{\alpha_l}(x))=\B_{\pi,\psi}(\mathrm{\bold{x}}_{\alpha_l}(x)t_n(a) \tilde{w}_n)=\B_{\pi,\psi}(t_n(a) \tilde{w}_n)(-\psi(x)/2)
$$
and so $t_n'=I_{2l}$ suffices. Thus, we see that the inclusion of the term $t_n'$ is necessary in the below theorem.

\begin{thm}\label{old GL_n}
Let $\pi$ and $\pi^\prime$ be irreducible cuspidal $\psi$-generic representations of $\SO_{2l}$ which share the same central character. If $\gamma(\pi\times\tau,\psi)=\gamma(\pi^\prime\times\tau,\psi)$ for all irreducible generic representations $\tau$ of $\GL_n$, then
$\B_{\pi,\psi}(t_n(a)  t_n' \tilde{w}_n)=\B_{\pi^\prime,\psi}(t_n(a) t_n'  \tilde{w}_n)$ for any $a\in\GL_n.$
\end{thm}

\begin{proof}
By Proposition \ref{GLnnonzero}, we have $\Psi(\B_{\pi,\psi},f_v)=\Psi(\B_{\pi^\prime,\psi},f_v)=W_v(I_n).$ From our assumption, $\gamma(\pi\times\tau,\psi)=\gamma(\pi^\prime\times\tau,\psi)$ for all irreducible generic representations $\tau$ of $\GL_n$, and hence it follows that $\Psi(\B_{\pi,\psi},\tilde{f_v})=\Psi(\B_{\pi^\prime,\psi},\tilde{f}_v).$ By definition,
$$\Psi(\B_{\pi,\psi},\tilde{f}_v)
=\sum_{g\in U_{\SO_{2n+1}}\setminus\SO_{2n+1}} \left(\sum_{r\in R^{l,n}} \B_{\pi,\psi}(r w^{l,n} g (w^{l,n})\inv)\right)   \tilde{f}_v(g, I_{n}).
$$
By Lemma \ref{intertwine},
$$\Psi(\B_{\pi,\psi},\tilde{f}_v)
=\sum_{\substack{a\in U_{GL_n}\setminus\GL_{n} \\ x\in V_n}} \left(\sum_{r\in R^{l,n}} \B_{\pi,\psi}(r w^{l,n} l_n(a) w_n x (w^{l,n})\inv)\right)   \tilde{f}_v(l_n(a) w_n x, I_{n}).
$$

The embedding of $\SO_{2n+1}$ into $\SO_{2l}$ takes $w_n$ to $t_n'\hat{w}_n$ (see \S\ref{BesselSupport}) and $l_n(a)$ to the element $q_n(a)=\mathrm{diag}(I_{l-n-1}, a, I_2, a^*, I_{l-n-1}).$
The embedding also takes the unipotent element $x=\left(\begin{matrix}
I_n & * & * \\
& 1 & * \\
& & I_n
\end{matrix}\right)$ to $\tilde{x}=\mathrm{diag}(I_{l-n-1},\left(\begin{matrix}
I_n & * & * \\
& I_2 & * \\
& & I_n
\end{matrix}\right), I_{l-n-1}),$ and we have
$$
\tilde{x}(w^{l,n})\inv=(w^{l,n})\inv\left(\begin{matrix}
I_n & 0 &  * & 0 & * \\
& I_{l-n-1} & 0 & 0 & 0\\
& & I_2 & 0 & * \\
& & & I_{l-n-1} & 0 \\
& & & & I_n
\end{matrix}\right).
$$
The character is trivial on the last unipotent matrix and hence, by Lemma \ref{intertwine},
$$\Psi(\B_{\pi,\psi},\tilde{f}_v)
=|V_n| \sum_{a\in U_{GL_n}\setminus\GL_{n}} \left(\sum_{r\in R^{l,n}} \B_{\pi,\psi}(r w^{l,n} q_n(a) t'_n \hat{w}_n (w^{l,n})\inv)\right)   W_v^*(a).
$$

Next, $w^{l,n} q_n(a)=t_n(a) w^{l,n}$ where $t_n(a)=\mathrm{diag}(a,I_{2l-2n}, a^*).$ Also, $w^{l,n}t_n'=t_n'w^{l,n}.$ Let $$r_x=\left(\begin{matrix}
I_n & & & & \\
x & I_{l-n-1} & & & \\
& & I_2 & & \\
& & & I_{l-n-1} & \\
& & & x^\prime & I_n
\end{matrix}\right)\in R^{l,n}.$$ Then, $r_x t_n(a) t_n'=t_n(a) t_n' r_{ax}.$ Recall that $\tilde{w}_n=w^{l,n}\hat{w}_n(w^{l,n})\inv.$ Then,
$$\Psi(\B_{\pi,\psi},\tilde{f}_v)
=|V_n| \sum_{a\in U_{GL_n}\setminus\GL_{n}} \left(\sum_{r_x\in R^{l,n}} \B_{\pi,\psi}(t_n(a) t_n' r_{ax} \tilde{w}_n)\right)   W_v^*(a).
$$
Note that
$$r_{ax}\tilde{w}_n=\tilde{w}_n\left(\begin{matrix}
I_n & &&  x'(a^*) & \\
& I_{l-n-1} & & & xa \\
& & I_2 & & \\
& & & I_{l-n-1} & \\
& & & & I_n
\end{matrix}\right).$$ The character is trivial on the last unipotent element and hence
$$\Psi(\B_{\pi,\psi},\tilde{f}_v)
=|V_n||R^{l,n}| \sum_{a\in U_{GL_n}\setminus\GL_{n}}  \B_{\pi,\psi}(t_n(a) t_n'  \tilde{w}_n)   W_v^*(a).
$$
Since $\Psi(\B_{\pi,\psi},\tilde{f_v})=\Psi(\B_{\pi^\prime,\psi},\tilde{f_v}),$ it follows that $$0=\sum_{a\in U_{GL_n}\setminus\GL_{n}}  (\B_{\pi,\psi}-\B_{\pi',\psi})(t_n(a) t_n'  \tilde{w}_n)   W_v^*(a).$$

Let $f$ be the function on $\GL_n$ defined by $f(a)=(\B_{\pi,\psi}-\B_{\pi',\psi})(t_n(a) t_n'  \tilde{w}_n).$ Then, $f(ua)=\psi(u)f(a)$ and $$0=\sum_{a\in U_{GL_n}\setminus\GL_{n}}  f(a)   W_v^*(a).$$ By Lemma \ref{Nien}, $f(a)=0$ for any $a\in\GL_n$ which proves the theorem.
 \end{proof}

As a consequence, we see that the twists by $\GL_n$ determine the Bessel functions on the Bruhat cells for $w\in\mathrm{B}_n(\SO_{2l}).$

\begin{thm}\label{GL_n}
Let $\pi$ and $\pi^\prime$ be irreducible cuspidal $\psi$-generic representations of $\SO_{2l}$ which share the same central character. If $\gamma(\pi\times\tau,\psi)=\gamma(\pi^\prime\times\tau,\psi)$ for all irreducible generic representations $\tau$ of $\GL_n$, then
$\B_{\pi,\psi}(b_1wb_2)=\B_{\pi^\prime,\psi}(b_1wb_2)$ for any $b_1,b_2\in B_{\SO_{2l}}$ and $w\in\mathrm{B}_n(\SO_{2l}).$
\end{thm}

\begin{proof}
We have $b_1wb_2=u_1twu_2$ for some $t\in T_{\SO_{2l}},$ and $u_1, u_2, \in U_{\SO_{2l}}.$ By Proposition \ref{Besselprop}, it suffices to show that $\B_{\pi,\psi}(tw)=\B_{\pi',\psi}(tw).$ We shall see that the support of these functions on $T_{\SO_{2l}}\mathrm{B}_n(\SO_{2l})$ is contained in the set $\{t_n(a)t_n'\tilde{w}_n , | , a\in\GL_n \}.$ Suppose that $\B_{\pi,\psi}(tw)\neq 0.$ By definition of $\mathrm{B}_n(\SO_{2l}),$ there exists $w'\in W(\GL_n)$ such that $w=t_n(w')\tilde{w}_n.$ Recall that \begin{align*}
    \tilde{w}_{n}\alpha_i &= \alpha_{n-i} \, \, \mathrm{for} \, \, 1\leq i \leq n-1,\\
    \tilde{w}_n\alpha_n (t) &= t_1\inv t_{n+1}\inv,\\
    \tilde{w}_n\alpha_i&=\alpha_i \, \, \mathrm{for} \, \, n+1\leq i \leq l-2,\\
    \tilde{w}_n \alpha_{l-1}&=\left\{\begin{array}{cc}
\alpha_l & \mathrm{if} \, n \, \mathrm{is} \,  \mathrm{odd,} \\
   \alpha_{l-1}  & \mathrm{if} \, n \, \mathrm{is} \,  \mathrm{even},
   \end{array}\right.\\
   \tilde{w}_n \alpha_{l}&=\left\{\begin{array}{cc}
\alpha_{l-1} & \mathrm{if} \, n \, \mathrm{is} \,  \mathrm{odd,} \\
   \alpha_{l}  & \mathrm{if} \, n \, \mathrm{is} \,  \mathrm{even.} 
\end{array}\right.
\end{align*}
Also $t_n(w')\alpha_i=\alpha_i$ for any $n+1 \leq i\leq l.$  Let $x\in\mathbb{F}_q$ and $\mathrm{\bold{x}}_{\alpha_i}(x)$ be an element of the root space of $\alpha_i.$ By Proposition \ref{Besselprop}, for $n+1 \neq i\leq l-2$ we have 
$$
\B_{\pi,\psi}(tw \mathrm{\bold{x}}_{\alpha_i}(x)) = \B_{\pi,\psi}(t w)\psi(x).
$$
On the other hand,
$$
\B_{\pi,\psi}(tw \mathrm{\bold{x}}_{\alpha_i}(x)) = \B_{\pi,\psi}(tw)\psi(\alpha_i(t)x).
$$
Since $\psi$ is nontrivial, we must have $\alpha_i(t)=1$ for any $n+1\leq i \leq l-2.$ That is, $t_i=t_{i+1}$ for any $n+1\leq i \leq l-2.$

Similarly,
$$
\B_{\pi,\psi}(t w \mathrm{\bold{x}}_{\alpha_l}(x)) = \B_{\pi,\psi}(t w)(-\psi(x)/2)
$$
and 
$$
\B_{\pi,\psi}(t w \mathrm{\bold{x}}_{\alpha_{l-1}}(x)) = \B_{\pi,\psi}(t w)(\psi(x)/4).
$$
When $n$ is odd,
$$
\B_{\pi,\psi}(t w \mathrm{\bold{x}}_{\alpha_l}(x)) = \B_{\pi,\psi}(\mathrm{\bold{x}}_{\alpha_{l-1}}(\alpha_{l-1}(t)x) tw)=\B_{\pi,\psi}( tw)(\psi(\alpha_{l-1}(t)x)/4),
$$
and
$$
\B_{\pi,\psi}(t w \mathrm{\bold{x}}_{\alpha_{l-1}}(x)) = \B_{\pi,\psi}(\mathrm{\bold{x}}_{\alpha_{l}}(\alpha_{l}(t)x) tw)=\B_{\pi,\psi}( tw)(-\psi(\alpha_{l}(t)x)/2),
$$
Since $\psi$ is nontrivial, $\frac{\alpha_{l-1}(t)}{4}=\frac{-1}{2}$ and  $\frac{-\alpha_{l}(t)}{2}=\frac{1}{4}.$ Set $t=\mathrm{diag}(t_1,\dots,t_l,t_l\inv,\dots,t_1\inv).$ We find that $t_{l-1}=\pm 1$ and $t_l=\mp \frac{1}{2}.$ Furthermore, $t_i=t_{i+1}=t_{l-1}=\pm1$ for any $n+1\leq i \leq l-2.$ Let $a=\mathrm{diag}(t_1,\dots,t_n)w'.$ Then $tw=t_n(a)\tilde{t}\tilde{w}_n$ if $t_{l-1}=1$ and by Theorem \ref{old GL_n}, we have $\B_{\pi,\psi}(tw)=\B_{\pi',\psi}(tw).$ If $t_{l-1}=-1,$ then $tw=-t_n(a)\tilde{t}\tilde{w}_n$. Since $\pi$ and $\pi'$ have the same central character, say $\chi$, by Theorem \ref{old GL_n}, \begin{align*}
    \B_{\pi,\psi}(tw) &= \B_{\pi,\psi}(-t_n(a)\tilde{t}\tilde{w}_n) \\ &=\chi(-1)\B_{\pi,\psi}(t_n(a)\tilde{t}\tilde{w}_n)\\
    &=\chi(-1)\B_{\pi',\psi}(t_n(a)\tilde{t}\tilde{w}_n)\\
    &=\B_{\pi',\psi}(-t_n(a)\tilde{t}\tilde{w}_n)\\
    &=\B_{\pi',\psi}(tw).
\end{align*}
This proves the corollary when $n$ is odd.

Suppose now that $n$ is even. We have
$$
\B_{\pi,\psi}(t w \mathrm{\bold{x}}_{\alpha_l}(x)) = \B_{\pi,\psi}(\mathrm{\bold{x}}_{\alpha_{l}}(\alpha_{l}(t)x) tw)=\B_{\pi,\psi}( tw)(\psi(\alpha_{l}(t)x)/4),
$$
and
$$
\B_{\pi,\psi}(t w \mathrm{\bold{x}}_{\alpha_{l-1}}(x)) = \B_{\pi,\psi}(\mathrm{\bold{x}}_{\alpha_{l-1}}(\alpha_{l-1}(t)x) tw)=\B_{\pi,\psi}( tw)(-\psi(\alpha_{l-1}(t)x)/2),
$$
Since $\psi$ is nontrivial, $\alpha_{l-1}(t)=1$ and  $\alpha_{l}(t)=1.$ Set $t=\mathrm{diag}(t_1,\dots,t_l,t_l\inv,\dots,t_1\inv).$ We find that $t_{l-1}=t_l=\pm 1.$ Furthermore, $t_i=t_{i+1}=t_{l-1}=\pm1$ for any $n+1\leq i \leq l-2.$ Let $a=\mathrm{diag}(t_1,\dots,t_n)w'.$ Then $tw=t_n(a)\tilde{w}_n$ if $t_{l-1}=1$ and by Theorem \ref{old GL_n}, we have $\B_{\pi,\psi}(tw)=\B_{\pi',\psi}(tw).$ If $t_{l-1}=-1,$ then $tw=-t_n(a)\tilde{w}_n$. Since $\pi$ and $\pi'$ have the same central character, say $\chi$, by Theorem \ref{old GL_n}, \begin{align*}
    \B_{\pi,\psi}(tw) &= \B_{\pi,\psi}(-t_n(a)\tilde{w}_n) \\ &=\chi(-1)\B_{\pi,\psi}(t_n(a)\tilde{w}_n)\\
    &=\chi(-1)\B_{\pi',\psi}(t_n(a)\tilde{w}_n)\\
    &=\B_{\pi',\psi}(-t_n(a)\tilde{w}_n)\\
    &=\B_{\pi',\psi}(tw).
\end{align*}
This proves the corollary when $n$ is even.
\end{proof}

\section{Twists by $\GL_{l-1}$}
In this section, we consider the twists of $\GL_{l-1}$. The arguments of this section are similar to the previous section, except that $R^{l,n}$ and $w^{l,n}$ are trivial which leads to minor differences.

Let $v\in\tau$ be a fixed vector and define $\xi_v\in I(\tau)$ by supp$(\xi_v)=L_{l-1} V_{l-1}= Q_{l-1}$ and 
$$
\xi_v(l_{l-1}(a)u)=\tau(a)v, \forall a\in\GL_{l-1}, u\in V_{l-1}.
$$
Let $f_v=f_{\xi_v}\in I(\tau,\psi^{-1}).$ That is, for $\Lambda_\tau\in \mathrm{Hom}_{U_{GL_{l-1}}}(\tau,\psi^{-1}), a\in\GL_{l-1},$ and $g\in\SO_{2l-1}$, we have $f_v(g,a)=\Lambda_\tau(\tau(a)\xi_v(g))$. We also let $W_v(a)=\Lambda_\tau(\tau(a)v).$ We show that the zeta integrals in this case are nonzero. 

\begin{prop}\label{GL_{l-1}nonzero}
Suppose $\B_{\pi,\psi}$ be the normalized Bessel function of $\pi$. Then, we have
$\Psi(\B_{\pi,\psi},f_v)= W_v(I_{l-1})$. In particular, we may choose $v\in\tau$ such that $\Psi(\B_{\pi,\psi},f_v)\neq 0.$
\end{prop}

\begin{proof}
By definition, $$\Psi(\B_{\pi,\psi},f_v)=\sum_{g\in U_{\SO_{2l-1}}\setminus \SO_{2l-1}} \B_{\pi,\psi}(g) f_v(g, I_{l-1}).$$
The support of $f_v(\cdot,I_{l-1})$ is $Q_{l-1}.$ Thus, $$\Psi(\B_{\pi,\psi},f_v)=\sum_{a\in U_{\GL_{l-1}}\setminus \GL_{l-1}} \B_{\pi,\psi}(l_{l-1}(a)) f_v(l_{l-1}(a), I_{l-1}).$$
Under the embedding of $\SO_{2l-1}$ into $\SO_{2l}$, $l_{l-1}(a)$ maps to $t_{l-1}(a).$ Hence, 
$$\Psi(\B_{\pi,\psi},f_v)=\sum_{a\in U_{\GL_{l-1}}\setminus \GL_{l-1}} \B_{\pi,\psi}(t_{l-1}(a)) W_v(a).$$
By Proposition \ref{uppertriangular}, $\B_{\pi,\psi}(t_{l-1}(a))=0$ unless $a$ is upper triangular. Thus, by Lemma $\ref{center}$,
$$\Psi(\B_{\pi,\psi},f_v)=W_v(I_{l-1}).$$ If we choose $v$ to be a Whittaker vector for $\tau$, then $W_v(I_{l-1})\neq 0.$ This proves the proposition.
 \end{proof}

The next theorem shows that the twists by $\GL_{l-1}$ determine that the Bessel functions are equal on a subset of the Bruhat cells in $\mathrm{B}_{l-1}(\SO_{2l}).$ More specifically, by Proposition \ref{Besselprop}, to determine the normalized Bessel functions on a Bruhat cell of $w\in\mathrm{B}_{l-1}(\SO_{2l})$, it is enough to determine it on  $T_{\SO_{2l}}w.$ The below theorem determines the subset of $T_{\SO_{2l}}w$ for $t\in T_{\SO_{2l}}$ such that the $l$-th coordinate is $t_l=1$ if $n$ is odd and $t_l=\frac{1}{2}$ if $n$ is even. Using the central character we can determine $t_l=\pm1$ if $n$ is odd and $t_l=\pm\frac{1}{2}$ if $n$ is even. Later, Theorem \ref{GL_l} shows that the twists by $\GL_l$ determine the remaining part where $t_l\neq\pm1$ if $n$ is odd and $t_l\neq\pm\frac{1}{2}$ if $n$ is even.

\begin{thm}\label{GL_{l-1}}
Let $\pi$ and $\pi^\prime$ be irreducible cuspidal $\psi$-generic representations of $\SO_{2l}$ which share the same central character. If $\gamma(\pi\times\tau,\psi)=\gamma(\pi^\prime\times\tau,\psi)$ for all irreducible generic representations $\tau$ of $\GL_{l-1}$, then we have
$\B_{\pi,\psi}(t_{l-1}(a)  t_{l-1}' \tilde{w}_{l-1})=\B_{\pi^\prime,\psi}(t_{l-1}(a) t_{l-1}'  \tilde{w}_{l-1})$ for any $a\in\GL_{l-1}.$
\end{thm}

\begin{proof}
By Proposition \ref{GL_{l-1}nonzero}, $\Psi(\B_{\pi,\psi},f_v)=\Psi(\B_{\pi^\prime,\psi},f_v)=W_v(I_{l-1}).$ From our assumption, $\gamma(\pi\times\tau,\psi)=\gamma(\pi^\prime\times\tau,\psi)$ for all irreducible generic representations $\tau$ of $\GL_{l-1}$, and hence it follows that $\Psi(\B_{\pi,\psi},\tilde{f_v})=\Psi(\B_{\pi^\prime,\psi},\tilde{f_v}).$ By definition,
$$\Psi(\B_{\pi,\psi},\tilde{f}_v)=\sum_{g\in U_{\SO_{2l-1}}\setminus \SO_{2l-1}} \B_{\pi,\psi}(g) \tilde{f}_v(g, I_{l-1}).$$
From Lemma \ref{intertwine}, we get
$$\Psi(\B_{\pi,\psi},\tilde{f}_v)=\sum_{\substack{a\in U_{\GL_{l-1}}\setminus \GL_{l-1} \\ u\in V_{l-1}}} \B_{\pi,\psi}(l_{l-1}(a)w_{l-1} u) W_v^*(a).$$
The embedding of $\SO_{2l-1}$ into $\SO_{2l}$ takes $l_{l-1}(a)$ to $t_{l-1}(a)$ and $w_{l-1}$ to $t_{l-1}'\tilde{w}_{l-1}.$ Let $$
u=\left(\begin{matrix}
I_{l-1} & x & y \\
& 1 & x' \\
& & I_{l-1}
\end{matrix}\right).
$$
Then the embedding takes $u$ to
$$
\tilde{u}=\left(\begin{matrix}
I_{l-1} & x & \frac{x}{2} & y \\
& 1 & & \frac{x'}{2} \\
& & 1 & x' \\
& & & I_{l-1}
\end{matrix}\right).
$$
Since $\psi(\tilde{u})=1,$ we have 
$$\Psi(\B_{\pi,\psi},\tilde{f}_v)=|V_{l-1}|\sum_{a\in U_{\GL_{l-1}}\setminus \GL_{l-1} } \B_{\pi,\psi}(t_{l-1}(a)t_{l-1}'\tilde{w}_{l-1} ) W_v^*(a).$$
From $\Psi(\B_{\pi,\psi},\tilde{f_v})=\Psi(\B_{\pi^\prime,\psi},\tilde{f_v}),$ it follows that $$0=\sum_{a\in U_{GL_{l-1}}\setminus\GL_{l-1}}  (\B_{\pi,\psi}-\B_{\pi',\psi})(t_{l-1}(a) t_{l-1}'  \tilde{w}_{l-1})   W_v^*(a).$$

Let $f$ be the function on $\GL_{l-1}$ defined by $f(a)=(\B_{\pi,\psi}-\B_{\pi',\psi})(t_{l-1}(a) t_{l-1}'  \tilde{w}_{l-1}).$ Then, $f(ua)=\psi(u)f(a)$ and $$0=\sum_{a\in\GL_{l-1}}  f(a)   W_v^*(a).$$ By Lemma \ref{Nien}, $f(a)=0$ for any $a\in\GL_{l-1}$ which gives the theorem.
 \end{proof}

The following corollary shows that Theorems \ref{GL_n} and \ref{GL_{l-1}} hold for the Bessel functions of $c\cdot\pi$ and $c\cdot\pi'$ as well.

\begin{cor}\label{cGL_n}
Let $\pi$ and $\pi^\prime$ be irreducible cuspidal $\psi$-generic representations of $\SO_{2l}$ which share the same central character. If $\gamma(\pi\times\tau,\psi)=\gamma(\pi^\prime\times\tau,\psi)$ for all irreducible generic representations $\tau$ of $\GL_{n}$ for any $n\leq l-1$, then
$\B_{c\cdot\pi,\psi}(t_{n}(a)  t_{n}' \tilde{w}_{n})=\B_{c\cdot\pi^\prime,\psi}(t_{n}(a) t_{n}'  \tilde{w}_{n})$ for any $a\in\GL_{n}.$
\end{cor}

\begin{proof}
By Theorems \ref{GL_n} and \ref{GL_{l-1}} ,
$\B_{\pi,\psi}(t_{n}(a)  t_{n}' \tilde{w}_{n})=\B_{\pi^\prime,\psi}(t_{n}(a) t_{n}'  \tilde{w}_{n})$ for any $a\in\GL_{n}.$ Also, 
$$\B_{\pi,\psi}(t_{n}(a) t_{n}'  \tilde{w}_{n})
=\B_{\pi,\psi}(c\tilde{t}\inv \tilde{t} ct_{n}(a) t_{n}'  \tilde{w}_{n} c\tilde{t}\inv \tilde{t} c)
=\B_{c\cdot\pi,\psi}( \tilde{t} ct_{n}(a) t_{n}'  \tilde{w}_{n} c\tilde{t}\inv)$$
$$
=\B_{c\cdot\pi,\psi}( t_{n}(a) t_{n}'  \tilde{w}_{n}).
$$ for any $n\leq l-1$. This proves the corollary.
 \end{proof}

The following corollary shows the equivalence of $\gamma$-factors of conjugate representations for $n\leq l-1$. This is not necessary to show the converse theorem, but when later paired with its analogue for $n=l$, Corollary \ref{conj l gamma}, it shows that the $\gamma$-factor cannot distinguish between a representation and its conjugate. In the setting of local fields, this is known \cite{Art, JL14}.

\begin{cor}\label{conj n gamma}
Let $\pi$ be an irreducible cuspidal $\psi$-generic representation of $\SO_{2l}$. Then we have $\gamma(\pi\times\tau,\psi)=\gamma(c\cdot\pi\times\tau,\psi)$ for all irreducible generic representations $\tau$ of $\GL_{n}$ for any $n\leq l-1$.
\end{cor}

\begin{proof}
We abuse notation and consider the zeta integrals simultaneously for each $n\leq l-1.$ By Propositions \ref{GLnnonzero} and \ref{GL_lnonzero} we can choose $v\in\tau$ such that $\Psi(\B_{\pi,\psi},f_v)=\Psi(\B_{c\cdot\pi,\psi},f_v)\neq 0.$ By the proofs of theorems \ref{GL_n} and \ref{GL_{l-1}},
$$\Psi(\B_{\pi,\psi},\tilde{f}_v)=|V_{n}||R^{l,n}|\sum_{a\in U_{\GL_{n}}\setminus \GL_{n} } \B_{\pi,\psi}(t_{n}(a)t_{n}'\tilde{w}_{n} ) W_v^*(a).$$ Since for any $n\leq l-1,$ we have
\begin{align*}
    \B_{\pi,\psi}(t_{n}(a) t_{n}'  \tilde{w}_{n})
&=\B_{\pi,\psi}(c\tilde{t}\inv \tilde{t} ct_{n}(a) t_{n}'  \tilde{w}_{n} c\tilde{t}\inv \tilde{t} c)\\
&=\B_{c\cdot\pi,\psi}( \tilde{t} ct_{n}(a) t_{n}'  \tilde{w}_{n} c\tilde{t}\inv)\\
&=\B_{c\cdot\pi,\psi}( t_{n}(a) t_{n}'  \tilde{w}_{n}),
\end{align*}
it follows that $\Psi(\B_{\pi,\psi},\tilde{f}_v)=\Psi(\B_{c\cdot\pi,\psi},\tilde{f}_v).$ Finally, since $\Psi(\B_{\pi,\psi},f_v)=\Psi(\B_{c\cdot\pi,\psi},f_v)$ is nonzero, we have $\gamma(\pi\times\tau,\psi)=\gamma(c\cdot\pi\times\tau,\psi)$ for all irreducible generic representations $\tau$ of $\GL_{n}$ for any $n\leq l-1$. This concludes the proof of the corollary.
 \end{proof}

\section{Twists by $\GL_l$}

In this section, we consider the twists by $\GL_l.$ These twists are significantly different from the previous twists. Note that we still need to determine the Bessel support on the Bruhat cells corresponding to Weyl elements in $\mathrm{B}_l(\SO_{2l})$ and $\mathrm{B}_l^c(\SO_{2l})$, and the Bessel support on the rest of Bruhat cells corresponding to Weyl elements in $\mathrm{B}_{l-1}(\SO_{2l})$ that are not determined by Theorem \ref{GL_{l-1}}. The difference from previous cases is that $c$ acts nontrivially on these sets.

Let $v\in\tau$ be a fixed vector and define $\xi_v\in I(\tau)$ by supp$(\xi_v)=L_{l} V_{l}= Q_{l}$ and 
$$
\xi_v(l_{l}(a)u)=\tau(a)v,$$ for any $a\in\GL_{l}, u\in V_{l}.
$
Let $f_v=f_{\xi_v}\in I(\tau,\psi^{-1}).$ That is $f_v(g,a)=\Lambda_\tau(\tau(a)\xi_v(g))$, where $\Lambda_\tau\in \mathrm{Hom}_{U_{GL_{l}}}(\tau,\psi^{-1}), a\in\GL_{l},$ and  $g\in\SO_{2l+1}$. Let $W_v(a)=\Lambda_\tau(\tau(a)v).$ The following proposition shows that the zeta integrals in this case are nonzero.

\begin{prop}\label{GL_lnonzero}
$\Psi(\B_{\pi,\psi},f_v)= W_v(w_{l,l})$. In particular, we may choose $v\in\tau$ such that $\Psi(\B_{\pi,\psi},f_v)\neq 0.$
\end{prop}

\begin{proof}
By definition.
$$
\Psi(\B_{\pi,\psi},f_v)=\sum_{g\in U_{\SO_{2l}}\setminus\SO_{2l}} \B_{\pi,\psi}(g)f_v(w_{l,l}g, I_{l}).
$$
By definition, the support of $f_v(\cdot, I_l)$ is $Q_l.$ Thus, $$
\Psi(\B_{\pi,\psi},f_v)=\sum_{a\in U_{\GL_{l}}\setminus\GL_{l}} \B_{\pi,\psi}(t_l(a))f_v(w_{l,l}l_l(a), I_{l}).
$$
By Proposition \ref{uppertriangular}, $\B_{\pi,\psi}(t_l(a))=0$ for any $a$ which isn't upper triangular. Thus, by Lemma \ref{center}, $\Psi(\B_{\pi,\psi},f_v)=W_v(w_{l,l}).$ If we let $\tilde{v}$ be a Whittaker vector for $\tau$ and $v=\tau(w_{l,l}\inv)\tilde{v},$ then $\Psi(\B_{\pi,\psi},f_v)=W_v(w_{l,l})=W_{\tilde{v}} (I_l)\neq 0.$ This proves the proposition.
 \end{proof}

The support of $\tilde{f}_v=M(\tau,\psi\inv)f_v$ is $Q_l w_l V_l.$ Recall that the twists by $\GL_n$ for $n\leq l-2$ determine the Bessel functions on the Bruhat cells for $\mathrm{B}_n(\SO_{2l}),$ while the twists by $\GL_{l-1}$ only determine a portion of the Bruhat cells for $\mathrm{B}_{l-1}(\SO_{2l}).$  We first show that the rest of the Bruhat cells, along with the Bruhat cells for $\mathrm{B}_{l}(\SO_{2l})$ and $\mathrm{B}_{l}^c(\SO_{2l})$, embed into $Q_l w_l V_l.$ This follows from the next three propositions.

\begin{prop}\label{l-1 twist embed}
Let $w\in\mathrm{B}_{l-1}(\SO_{2l})$ and $t=\mathrm{diag}(t_1,\dots, t_l, t_l\inv,\dots,t_1\inv) \in T_{\SO_{2l}}$. Then $tw\in Q_l w_l V_l$ (via the embedding of $\SO_{2l}$ into $\SO_{2l+1}$) if and only if $t_l\neq 1$ if $l$ is odd, or $t_l\neq \frac{-1}{2}$ if $l$ is even.
\end{prop}

\begin{proof}
We begin with a sketch of the proof. 
We first compute the embedding of the torus in coordinates and the embedding of the Weyl element $w$, then compute $Q_l w_l V_l$ in terms of block matrices. We see that the lower left corner is determined by the torus. In particular, this lower left corner must be invertible if $tw\in Q_l w_l V_l$ and the determinant of this block gives us the condition we need. We proceed with the details of the proof.

Recall that the embedding sends $t=\mathrm{diag}(t_1,\dots, t_l, t_l\inv,\dots,t_1\inv)$ to
$$
\mathrm{diag}(s,\left(\begin{matrix}
\frac{1}{2}+\frac{1}{4}(t_l+t_l\inv) & \frac{1}{2}(t_l-t_l\inv) & 2(\frac{1}{2}-\frac{1}{4}(t_l+t_l\inv)) \\
\frac{1}{4}(t_l-t_l\inv) & \frac{1}{2}(t_l+t_l\inv) & \frac{-1}{2}(t_l-t_l\inv) \\
\frac{1}{2}(\frac{1}{2}-\frac{1}{4}(t_l+t_l\inv)) & \frac{-1}{4}(t_l-t_l\inv) & \frac{1}{2}+\frac{1}{4}(t_l+t_l\inv)
\end{matrix}\right), s^*),
$$
where $s=\mathrm{diag}(t_1, t_2, \dots, t_{l-1})$.
By Proposition \ref{l-1 Besselpart}, there exists $w^\prime\in W(\GL_l)$ such that $w=t_l(w^\prime)\tilde{w}_l$ and  
$$
w^\prime=\left(\begin{matrix}
& w'' \\
1 & 
\end{matrix}\right),
$$
for some $w''\in W(\GL_{l-1}).$ Then, 
$$
w=\left(\begin{matrix}
&  & w'' \\
& J_2^{l-1} & \\
(w'')^* & & 
\end{matrix}\right).
$$
Thus, the embedding into $\SO_{2l+1}$ of $w$ is
$$
\left(\begin{matrix}
 & & w'' \\
& A^{l-1} & \\
(w'')^* & & 
\end{matrix}
\right),
$$
where $$
A=\left(\begin{matrix}
\frac{-1}{8} & \frac{3}{4} & \frac{9}{4} \\
\frac{-3}{8} & \frac{5}{4} & \frac{3}{4} \\
\frac{9}{16} & \frac{-3}{8} & \frac{-1}{8}
\end{matrix}
\right).
$$
Note that $A^{l-1}=I_3$ when $l$ is odd. 

Thus, we can compute the embedding of $tw$ in $\SO_{2l+1}$. We obtain
$$
\left(\begin{matrix}
& & \mathrm{diag}(t_1,\dots,t_{l-1})w'' \\
& Z & \\
\mathrm{diag}(t_{l-1}\inv,\dots,t_{1}\inv)(w'')^*
\end{matrix}\right),
$$
where 
$$Z=\left(\begin{matrix}
\frac{1}{2}+\frac{1}{4}(t_l+t_l\inv) & \frac{1}{2}(t_l-t_l\inv) & 2(\frac{1}{2}-\frac{1}{4}(t_l+t_l\inv)) \\
\frac{1}{4}(t_l-t_l\inv) & \frac{1}{2}(t_l+t_l\inv) & \frac{-1}{2}(t_l-t_l\inv) \\
\frac{1}{2}(\frac{1}{2}-\frac{1}{4}(t_l+t_l\inv)) & \frac{-1}{4}(t_l-t_l\inv) & \frac{1}{2}+\frac{1}{4}(t_l+t_l\inv)
\end{matrix}\right) A^{l-1}.$$

Next, we turn towards computing $Q_l w_l V_l.$ Let $a\in \GL_l$ and $n_1,n_2\in V_l$. We wish to compute $l_l(a)n_1 w_l n_2.$ Write 
$$
n_1=\left(\begin{matrix}
I_l & X & Y \\
    & 1 & X^\prime \\
    &   & I_l
\end{matrix}\right),
$$
where $X\in\mathbb{F}_q^l$, $Y\in\mathrm{Mat}_{l\times l}(\mathbb{F}_q),$ $X^\prime=-{}^t X J_l$, and $J_l Y + {}^t X^\prime\cdot X^\prime + {}^t Y J_l =0.$
Similarly, we write$$
n_2=\left(\begin{matrix}
I_l & M & N \\
    & 1 & M^\prime \\
    &   & I_l
\end{matrix}\right).
$$
Then, 
$$l_l(a)n_1 w_l n_2=
\left(\begin{matrix}
aY & aX+aYM & a+aX M^\prime+aYN \\
X^\prime    & 1+X^\prime M & M^\prime+X^\prime N \\
a^*    & a^*M  & a^*N
\end{matrix}\right).
$$

Thus, $tw=l_l(a)n_1 w_l n_2$ has a solution as long as we can find a suitable $a^*.$ We see that $a^*=\left(\begin{matrix}
& Z_{3,1} \\
\mathrm{diag}(t_{l-1}\inv,\dots,t_{1}\inv)(w'')^* &
\end{matrix}\right)$ where $Z_{3,1}$ is the $(3,1)$-entry of $Z$. Such $a^*$ exists as long as this matrix is invertible. Hence we need $Z_{3,1}\neq 0.$

If $l$ is odd, then $Z_{3,1}=\frac{1}{2}(\frac{1}{2}-\frac{1}{4}(t_l+t_l\inv)).$ This is nonzero as long as $t_l\neq 1$ and hence the proposition follows.

If $l$ is even then $Z_{3,1}=\frac{1}{4}+\frac{1}{4}t_l+\frac{1}{16}t_l\inv.$ This is nonzero as long as $t_l\neq \frac{-1}{2}$ and hence the proposition follows. This finishes the proof of the proposition.
 \end{proof}

Next we examine the Bruhat cells for $\mathrm{B}_{l}(\SO_{2l})$ and $\mathrm{B}_{l}^c(\SO_{2l}).$

\begin{prop}\label{l twist supp}
Let $w\in\mathrm{B}_{l}(\SO_{2l})$. Then, for any $t\in T_{\SO_{2l}}$, we have $tw\in Q_l w_l V_l$ under the embedding of $\SO_{2l}$ into $\SO_{2l+1}.$
\end{prop}

\begin{proof}
Since $w\in\mathrm{B}_{l}(\SO_{2l})$, $w\notin\mathrm{B}_{l-1}(\SO_{2l}).$ By Proposition \ref{l-1 Besselpart}, there exists $w^\prime\in W(\GL_l)$ such that $w=t_l(w^\prime)\tilde{w}_l$ where $w'=(w_{i,j}')_{i,j=1}^l\in W(\GL_l)$ with $w_{l,1}'=0.$ Similarly, we also write $(w')^*=(w_{i,j}'^*)_{i,j=1}^l.$ 

As in the previous proof, we compute $Q_l w_l V_l.$ Let $a\in \GL_l$ and $n_1,n_2\in V_l$. We wish to compute $l_l(a)n_1 w_l n_2.$ Write 
$$
n_1=\left(\begin{matrix}
I_l & X & Y \\
    & 1 & X^\prime \\
    &   & I_l
\end{matrix}\right),
$$
where $X\in\mathbb{F}_q^l$, $Y\in\mathrm{Mat}_{l\times l}(\mathbb{F}_q),$ $X^\prime=-{}^t X J_l$, and $J_l Y + {}^t X^\prime\cdot X^\prime + {}^t Y J_l =0.$
Similarly, we write$$
n_2=\left(\begin{matrix}
I_l & M & N \\
    & 1 & M^\prime \\
    &   & I_l
\end{matrix}\right).
$$
Then, 
$$l_l(a)n_1 w_l n_2=
\left(\begin{matrix}
aY & aX+aYM & a+aX M^\prime+aYN \\
X^\prime    & 1+X^\prime M & M^\prime+X^\prime N \\
a^*    & a^*M  & a^*N
\end{matrix}\right).
$$

We proceed by computing the image of $w$ under the embedding given by Equation \ref{equ emb n=l}. Recall that we defined the matrix
$$
\tilde{M}=\left(\begin{matrix}
    \frac{1}{4} & \frac{1}{2} & \frac{-1}{2} \\
    \frac{1}{2} & 0 & 1 \\
    \frac{-1}{2} & 1 & 1
    \end{matrix}\right)
$$ which was used in the embedding. The image of the embedding depends on the parity of $l$. 

Let $l$ be even. Then the embedding of $\SO_{2l}$ into $\SO_{2l+1}$ takes $w$ to

$$
\left(\begin{matrix}
0_{(l-1)\times(l-1)} & \left(\begin{matrix}
0 & 0 & (w_{i,1}')_{i=1}^{l-1}
\end{matrix}\right) \tilde{M} & (w_{i,j}')_{i=1,j=2}^{i=l-1,j=l} \\
\tilde{M}\inv \left(\begin{matrix}
0 \\
0 \\
(w_{1,j}'^*)_{j=1}^{l-1}  
\end{matrix}\right)
& \tilde{M}\inv \left(\begin{matrix}
0 & 0 & 0 \\
0 & 1 & 0 \\
0 & 0 & 0
\end{matrix}\right)\tilde{M} & \tilde{M}\inv \left(\begin{matrix}
(w_{l,j}')_{j=2}^l   \\
0 \\
0
\end{matrix}\right)  \\
(w_{i,j}'^*)_{i=2,j=1}^{i=l,j=l-1}  & \left(\begin{matrix}
(w_{i,l}'^*)_{i=2}^l & 0 & 0
\end{matrix}\right) \tilde{M} & 0_{(l-1)\times(l-1)}
\end{matrix}\right).
$$
Thus, $tw=l_l(a)n_1 w_l n_2$ as long as we can find $a$ such that
$$
a^*=\left(\begin{matrix}
\frac{t_l\inv(w_{1,j}'^*)_{j=1}^{l-1}}{4} & \frac{1}{4} \\
\mathrm{diag}(t_{l-1}\inv,\dots, t_1\inv) (w_{i,j}'^*)_{i=2,j=1}^{i=l,j=l-1} & \mathrm{diag}(t_{l-1}\inv,\dots, t_1\inv)\frac{(w_{i,l}'^*)_{i=2}^l}{4} \\
\end{matrix}\right).
$$
This is possible as the determinant of this matrix is nonzero. Indeed, the determinant of this matrix is (up to a sign) $\frac{1}{16} t_l\inv t_{l-1}\inv \cdots t_1\inv$ since $(w')^*\in W(\GL_l)$ and $(w_{i,j}'^*)_{i=2,j=1}^{i=l,j=l-1}$ contains a row of zeroes. Therefore, $tw\in Q_l w_l V_l$ for any $t\in T_{\SO_{2l}}.$

The proof for $l$ odd is similar. We obtain that $tw=l_l(a)n_1 w_l n_2$ as long as we can find $a$ such that
$$
a^*=\left(\begin{matrix}
\frac{t_l\inv(w_{1,j}'^*)_{j=1}^{l-1}}{4} & \frac{1}{4} \\
\mathrm{diag}(t_{l-1}\inv,\dots, t_1\inv) (w_{i,j}'^*)_{i=2,j=1}^{i=l,j=l-1} & \mathrm{diag}(t_{l-1}\inv,\dots, t_1\inv)\frac{-(w_{i,l}'^*)_{i=2}^l}{2} \\
\end{matrix}\right).
$$
Again, this is always possible as the determinant of this matrix is (up to a sign) $\frac{1}{8} t_l\inv\cdots t_1\inv.$
Therefore, $tw\in Q_l w_l V_l$ for any $t\in T_{\SO_{2l}}$ which concludes the proof of the proposition.
 \end{proof}

\begin{prop}\label{l twist conj supp}
Let $w\in\mathrm{B}_{l}^c(\SO_{2l})$. Then for any $t\in T_{\SO_{2l}}$ $tw\in Q_l w_l V_l$ under the embedding of $\SO_{2l}$ into $\SO_{2l+1}.$
\end{prop}

\begin{proof}
The proof is similar to that of the previous proposition. We record the $a^*$'s for future reference. 

For $l$ even, we have $tw=l_l(a)n_1 w_l n_2$ where $a\in\GL_l$ and $n_1,n_2\in V_l$ if 
$$
a^*=\left(\begin{matrix}
\frac{-t_l(w_{1,j}'^*)_{j=1}^{l-1}}{2} & \frac{1}{4} \\
\mathrm{diag}(t_{l-1}\inv,\dots, t_1\inv) (w_{i,j}'^*)_{i=2,j=1}^{i=l,j=l-1} & \mathrm{diag}(t_{l-1}\inv,\dots, t_1\inv)\frac{-(w_{i,l}'^*)_{i=2}^l}{2} \\
\end{matrix}\right).
$$

For $l$ odd, we have $tw=l_l(a)n_1 w_l n_2$ where $a\in\GL_l$ and $n_1,n_2\in V_l$ if 
$$
a^*=\left(\begin{matrix}
\frac{-t_l(w_{1,j}'^*)_{j=1}^{l-1}}{2} & \frac{1}{4} \\
\mathrm{diag}(t_{l-1}\inv,\dots, t_1\inv) (w_{i,j}'^*)_{i=2,j=1}^{i=l,j=l-1} & \mathrm{diag}(t_{l-1}\inv,\dots, t_1\inv)\frac{(w_{i,l}'^*)_{i=2}^l}{4} \\
\end{matrix}\right).
$$
 \end{proof}

The next two lemmas relate Bessel function for $\pi$ and $c\cdot\pi.$ In particular, the below lemma shows that the contribution of the zeta integral from $\mathrm{B}_l^c(\SO_{2l})$ for $\pi$ can be described using an integral over $\mathrm{B}_l(\SO_{2l})$ and the Bessel function for $c\cdot\pi.$

\begin{lemma}\label{l contribution to GL_l}
Let $\pi$ be an irreducible cuspidal $\psi$-generic representations of $\SO_{2l}.$ Then we have
\begin{align*}
    &\sum_{t\in T_{\SO_{2l}}, w\in \mathrm{B}_l^c(\SO_{2l}), u\in U_{\SO_{2l}}}
\B_{\pi,\psi}(twu)\tilde{f}_v(w_{l,l} twu, I_l) \\
=&\sum_{t\in T_{\SO_{2l}}, w\in \mathrm{B}_l(\SO_{2l}), u\in U_{\SO_{2l}}}
\B_{c\cdot\pi,\psi}(twu)
\tilde{f}_v(w_{l,l} twu, I_l).
\end{align*}
\end{lemma}

\begin{proof}
We write Let $t=\mathrm{diag}(t_1,\dots,t_{l-1},t_l, t_l\inv,t_{l-1}\inv,\dots,t_1\inv).$
By Proposition \ref{Besselprop}, we have $\B_{\pi,\psi}(twu)=\B_{\pi,\psi}(tw)\psi(u)$ for any $t\in T_{\SO_{2l}},$ $u\in U_{\SO_{2l}},$ and $w\in W(\SO_{2l}).$ Also,
\begin{align*}
 \B_{\pi,\psi}(tw)
=\B_{\pi,\psi}(c\tilde{t}\inv \tilde{t} ctw c\tilde{t}\inv \tilde{t} c)
=\B_{c\cdot\pi,\psi}( \tilde{t} ctw c\tilde{t}\inv)
=\B_{c\cdot\pi,\psi}( t' cwc),
\end{align*}
where $t'=\tilde{t}ctc (cwc\tilde{t}\inv (cwc)\inv).$ For $w\in\mathrm{B}_l^c(\SO_{2l})$, we have $cwc\in \mathrm{B}_l(\SO_{2l}).$ By Proposition \ref{l-1 Besselpart}, there exists $w^\prime\in W(\GL_l)$ such that $cwc=t_l(w^\prime)\tilde{w}_l$ where $w'=(w_{i,j}')_{i,j=1}^l\in W(\GL_l)$ with $w_{l,1}'=0.$ We also let $(w')^*=(w_{i,j}'^*)_{i,j=1}^l$ and $r$ be such that $w_{r,1}'=1.$ 

For $l$ even, $(cwc\tilde{t}\inv (cwc)\inv)=t_l(\mathrm{diag}(1,\dots, 1, -2, 1,\dots, 1))$ where $-2$ is the $r$-th entry. Thus, $t'=t_l(\mathrm{diag}(t_1,\dots, t_{r-1}, -2t_r, t_{r+1},\dots, t_{l-1}, \frac{-t_l\inv}{2})).$ Let $a_1,a_2\in\GL_l$ and $n_1,n_2,n_3,n_4\in V_l$ be such that 
$tw=l_l(a_1)n_1 w_l n_2$ and $tcwc=l_l(a_2)n_3 w_l n_4.$ By the proofs of Proposition \ref{l twist supp} and \ref{l twist conj supp}, we have
$$
a_1^*=\left(\begin{matrix}
\frac{-t_l(w_{1,j}'^*)_{j=1}^{l-1}}{2} & \frac{1}{4} \\
\mathrm{diag}(t_{l-1}\inv,\dots, t_1\inv) (w_{i,j}'^*)_{i=2,j=1}^{i=l,j=l-1} & \mathrm{diag}(t_{l-1}\inv,\dots, t_1\inv)\frac{-(w_{i,l}'^*)_{i=2}^l}{2} \\
\end{matrix}\right),$$ and
$$
a_2^*=\left(\begin{matrix}
\frac{t_l\inv(w_{1,j}'^*)_{j=1}^{l-1}}{4} & \frac{1}{4} \\
\mathrm{diag}(t_{l-1}\inv,\dots, t_1\inv) (w_{i,j}'^*)_{i=2,j=1}^{i=l,j=l-1} & \mathrm{diag}(t_{l-1}\inv,\dots, t_1\inv)\frac{(w_{i,l}'^*)_{i=2}^l}{4} \\
\end{matrix}\right).
$$
Performing the change of variables $t_r\mapsto \frac{-t_r}{2}$ and $t_l\mapsto \frac{-t_l\inv}{2}$ takes $t'\mapsto t$ and $a_1^*\mapsto a_2^*.$ Note that $w'^*_{l+1-r,l}=1$ and this is the coordinate for $t_r\inv.$ Thus we obtain the claim for $l$ even:
\begin{align*}
    &\quad\sum_{t\in T_{\SO_{2l}}, w\in \mathrm{B}_l^c(\SO_{2l}), u\in U_{\SO_{2l}}}
\B_{\pi,\psi}(twu)\tilde{f}_v(w_{l,l} twu, I_l)\\
&=\sum_{t\in T_{\SO_{2l}}, w\in \mathrm{B}_l^c(\SO_{2l}), u\in U_{\SO_{2l}}}
\B_{c\cdot\pi,\psi}( t' cwcu)
\tilde{f}_v(w_{l,l} twu, I_l)\\
&=\sum_{t\in T_{\SO_{2l}}, w\in \mathrm{B}_l(\SO_{2l}), u\in U_{\SO_{2l}}}
\B_{c\cdot\pi,\psi}( t wu)
\tilde{f}_v(w_{l,l} twu, I_l).
\end{align*}

For $l$ is odd, we obtain a different change of variables which matches the $a^*$'s in the odd cases of Propositions \ref{l twist supp} and \ref{l twist conj supp}. We omit the details. Again, we find that
\begin{align*}
&\quad\sum_{t\in T_{\SO_{2l}}, w\in \mathrm{B}_l^c(\SO_{2l}), u\in U_{\SO_{2l}}}
\B_{\pi,\psi}(twu)\tilde{f}_v(w_{l,l} twu, I_l)\\
&=\sum_{t\in T_{\SO_{2l}}, w\in \mathrm{B}_l(\SO_{2l}), u\in U_{\SO_{2l}}}
\B_{c\cdot\pi,\psi}( t wu)
\tilde{f}_v(w_{l,l} twu, I_l).
\end{align*} 
\end{proof}

The contribution of $\mathrm{B}_{l-1}(\SO_{2l})$ to the zeta integral is more nuanced. Since $cwc=w$ for $w\in\mathrm{B}_{l-1}(\SO_{2l}),$ one has to split the integral in half and then conjugate one half in order to obtain an integral involving a sum of Bessel functions for $\pi$ and $c\cdot\pi.$ Note that from Theorem \ref{GL_{l-1}} it follows that we need to only consider the subset of the torus defined by $$T_l=\{t=\mathrm{diag}(t_1,\dots,t_l,t_l\inv, \dots, t_1\inv)\in T_{\SO_{2l}} \, | \, t_l\neq \pm c_l\},$$ where $c_l= 1$ if $l$ is odd or $\frac{1}{2}$ if $l$ is even.
\begin{lemma}\label{l-1 contribution to GL_l}
Let $\pi$ be an irreducible cuspidal $\psi$-generic representation of $\SO_{2l}$.
There exists a subset $A_l\subseteq T_l,$ specified in the proof, such that
\begin{align*}
&\sum_{t\in T_l, w\in \mathrm{B}_{l-1}(\SO_{2l}), u\in U_{\SO_{2l}}}
\B_{\pi,\psi}(twu)\tilde{f}_v(w_{l,l} twu, I_l) \\
=&\sum_{t\in A_l, w\in \mathrm{B}_{l-1}(\SO_{2l}), u\in U_{\SO_{2l}}}
(\B_{\pi,\psi}+\B_{c\cdot\pi,\psi})(twu)\tilde{f}_v(w_{l,l} twu, I_l).
\end{align*}
\end{lemma}

\begin{proof}
Let $t=\mathrm{diag}(t_1,\dots,t_{l-1},t_l, t_l\inv,t_{l-1}\inv,\dots,t_1\inv).$ We have
\begin{align*}
    \B_{\pi,\psi}(tw)=\B_{\pi,\psi}(c\tilde{t}\inv \tilde{t} ctw c\tilde{t}\inv \tilde{t} c)=\B_{c\cdot\pi,\psi}( \tilde{t} ctw c\tilde{t}\inv)
 =\B_{c\cdot\pi,\psi}( t' cwc),
\end{align*}
where $t'=\tilde{t}ctc (cwc\tilde{t}\inv (cwc)\inv).$ For $w\in \mathrm{B}_{l-1}(\SO_{2l}),$ $cwc=w.$ Hence $t'=\tilde{t}ctcw\tilde{t}\inv w\inv$ and 
$
\B_{c\cdot\pi,\psi}( t' cwc)=\B_{c\cdot\pi,\psi}( t' w).$

If $l$ is even, then $w\tilde{t}\inv w\inv=\tilde{t}$ and so $t'=\mathrm{diag}(t_1,\dots,t_{l-1},\frac{t_l\inv}{4}, 4t_l,t_{l-1}\inv,\dots,t_1\inv).$ The map $t_l\mapsto \frac{t_l\inv}{4}$ sends $t^\prime$ to $t$ and $\frac{1}{4}+\frac{1}{4}t_l+\frac{1}{16}t_l\inv$ to itself. The fixed points of the map $t_l\mapsto \frac{t_l\inv}{4}$ are $t_l=\pm\frac{1}{2}.$ Furthermore, the mapping is an involution. Therefore, we can partition $\mathbb{F}_q^\times\setminus\{\pm\frac{1}{2}\}$ into two disjoint sets $A$ and $B$ such that if $t_l\in A$ then $\frac{t_l\inv}{4}\in B.$ Hence, we partition $T_l$ into two sets $A_l$ and $B_l$ such that $t\in A_l$ if and only if $t_l\in A.$ Note that in this case $tw=l_l(a)n_1 w_l n_2$ for $a\in \GL_l$ and $n_1,n_2\in V_l$ with
$$
a^*=\left(\begin{matrix}
& \frac{1}{4}+\frac{1}{4}t_l+\frac{1}{16}t_l\inv \\
\mathrm{diag}(t_{l-1}\inv,\dots,t_{1}\inv)(w'')^* &
\end{matrix}\right).
$$
The map $t_l\mapsto \frac{t_l\inv}{4}$ takes $a^*\mapsto a^*.$ Therefore,
\begin{align*}
&\,\,\,\,\,\,\,\sum_{t\in T_l, w\in \mathrm{B}_{l-1}(\SO_{2l}), u\in U_{\SO_{2l}}}
\B_{\pi,\psi}(twu)\tilde{f}_v(w_{l,l} twu, I_l)\\
&=\sum_{t\in A_l, w\in \mathrm{B}_{l-1}(\SO_{2l}), u\in U_{\SO_{2l}}}
\B_{\pi,\psi}(twu)\tilde{f}_v(w_{l,l} twu, I_l)\\
&\quad+\,\sum_{t\in B_l, w\in \mathrm{B}_{l-1}(\SO_{2l}), u\in U_{\SO_{2l}}}
\B_{\pi,\psi}(twu)\tilde{f}_v(w_{l,l} twu, I_l)\\
&=\sum_{t\in A_l, w\in \mathrm{B}_{l-1}(\SO_{2l}), u\in U_{\SO_{2l}}}
(\B_{\pi,\psi}+\B_{c\cdot\pi,\psi})(twu)\tilde{f}_v(w_{l,l} twu, I_l).
\end{align*}

The case for $l$ odd is similar. We have $w\tilde{t}\inv w\inv=\tilde{t}\inv$ and hence $t'=ctc.$ The map $t_l\mapsto t_l\inv$ sends $t^\prime$ to $t$ and $\frac{1}{2}(\frac{1}{2}-\frac{1}{4}(t_l+t_l\inv))$ to itself. The fixed points of the map $t_l\mapsto t_l\inv$ are $t_l=\pm 1.$ Furthermore, the mapping is an involution. Therefore, we can partition $\mathbb{F}_q^\times\setminus\{\pm 1\}$ into two disjoint sets $A$ and $B$ such that if $t_l\in A$ then $t_l\inv\in B.$ Furthermore we partition $T_l$ into two sets $A_l$ and $B_l$ such that $t\in A_l$ if and only if $t_l\in A.$ Note that in this case $tw=l_l(a)n_1 w_l n_2$ for $a\in \GL_l$ and $n_1,n_2\in V_l$ with
$$
a^*=\left(\begin{matrix}
& \frac{1}{2}(\frac{1}{2}-\frac{1}{4}(t_l+t_l\inv)) \\
\mathrm{diag}(t_{l-1}\inv,\dots,t_{1}\inv)(w'')^* &
\end{matrix}\right).
$$
The map $t_l\mapsto t_l\inv$ takes $a^*\mapsto a^*.$ As in the previous case, we find
\begin{align*}
&\quad\sum_{t\in T_l, w\in \mathrm{B}_{l-1}(\SO_{2l}), u\in U_{\SO_{2l}}}
\B_{\pi,\psi}(twu)\tilde{f}_v(w_{l,l} twu, I_l)\\
&=\sum_{t\in A_l, w\in \mathrm{B}_{l-1}(\SO_{2l}), u\in U_{\SO_{2l}}}
(\B_{\pi,\psi}+\B_{c\cdot\pi,\psi})(twu)\tilde{f}_v(w_{l,l} twu, I_l).
\end{align*}
Hence we have proven the lemma for any $l$.
\end{proof}

We are ready to compute the sets determined by the twists by $\GL_n$, $n\leq l$. We remark that it is possible to prove the below theorem assuming that the $\gamma$-factors are equal only for twists by $\GL_l.$ Indeed, computations of the embedding of $\SO_{2l}$ into $\SO_{2l+1}$ show that any element of a Bruhat cell corresponding to $\mathrm{B}_n(\SO_{2l})$ for $n\leq l-2$ does not embed into the support of $\tilde{f}_v$ (and we obtain a similar result for the set determined by Theorem \ref{GL_{l-1}}). However, this is not necessary for the converse theorem. Instead, we use the results of the previous sections to eliminate the contributions from Bruhat cells corresponding to $\mathrm{B}_n(\SO_{2l})$ for $n\leq l-2$ (and the set determined by Theorem \ref{GL_{l-1}}).

\begin{thm}\label{GL_l}
Let $\pi$ and $\pi^\prime$ be irreducible cuspidal $\psi$-generic representations of $\SO_{2l}$ with the same central character. If $\gamma(\pi\times\tau,\psi)=\gamma(\pi^\prime\times\tau,\psi)$ for all irreducible generic representations $\tau$ of $\GL_n$ where $n\leq l,$
then 
\begin{enumerate}
    \item[(1)] $$(\B_{\pi,\psi}+\B_{c\cdot\pi,\psi})(tw)=(\B_{\pi^\prime,\psi}+\B_{c\cdot\pi^\prime,\psi})(tw),$$ for any $t\in T_{\SO_{2l}}$ and $w\in \mathrm{B}_l(\SO_{2l})\cup\mathrm{B}_l^c(\SO_{2l})$;  
    \item[(2)]  $$(\B_{\pi,\psi}+\B_{c\cdot\pi,\psi})(tw)=(\B_{\pi^\prime,\psi}+\B_{c\cdot\pi^\prime,\psi})(tw),$$ for any $w\in\mathrm{B}_{l-1}(\SO_{2l})$ and $t\in T_{\SO_{2l}}$ with $t_l\neq \pm 1$ if $l$ is odd or $t_l\neq \pm\frac{1}{2}$ if $l$ is even.
\end{enumerate}
\end{thm}

\begin{proof}
By Proposition \ref{GL_lnonzero}, we have $\Psi(\B_{\pi,\psi},f_v)=\Psi(\B_{\pi',\psi},f_v)=W_v(w_{l,l}).$ By assumption, $\gamma(\pi\times\tau,\psi)=\gamma(\pi^\prime\times\tau,\psi)$ for all irreducible generic representations $\tau$ of $\GL_l$, and hence we have that $\Psi(\B_{\pi,\psi},\tilde{f}_v)=\Psi(\B_{\pi',\psi},\tilde{f}_v).$
By definition,
$$
\Psi(\B_{\pi,\psi},\tilde{f}_v)=\sum_{g\in U_{\SO_{2l}}\setminus\SO_{2l}} \B_{\pi,\psi}(g)\tilde{f}_v(w_{l,l}g, I_{l}),
$$ and hence
$$
0=\sum_{g\in U_{\SO_{2l}}\setminus\SO_{2l}}( \B_{\pi,\psi}-\B_{\pi',\psi})(g)\tilde{f}_v(w_{l,l}g, I_{l}).
$$
Let $T_l=\{t=\mathrm{diag}(t_1,\dots,t_l,t_l\inv, \dots, t_1\inv)\in T_{\SO_{2l}} \, | \, t_l\neq \pm c_l\}$ where $c_l= 1$ if $l$ is odd or $\frac{1}{2}$ if $l$ is even.
Since $\gamma(\pi\times\tau,\psi)=\gamma(\pi^\prime\times\tau,\psi)$ for all irreducible generic representations $\tau$ of $\GL_n$ with $n\leq l-1$, by Propositions \ref{l-1 twist embed}, \ref{l twist supp}, \ref{l twist conj supp}, and Theorems \ref{GL_n} and \ref{GL_{l-1}}, 
\begin{align*}
0&=\sum_{t\in T_{\SO_{2l}}, w\in \mathrm{B}_l(\SO_{2l}), u\in U_{\SO_{2l}}}
(\B_{\pi,\psi}-\B_{\pi^\prime,\psi})(twu)\tilde{f}_v(w_{l,l}twu, I_l)\\
&\quad+\sum_{t\in T_{\SO_{2l}}, w\in \mathrm{B}_l^c(\SO_{2l}), u\in U_{\SO_{2l}}}
(\B_{\pi,\psi}-\B_{\pi^\prime,\psi})(twu)\tilde{f}_v(w_{l,l} twu, I_l)\\
&\quad+\sum_{t\in T_l, w\in \mathrm{B}_{l-1}(\SO_{2l}), u\in U_{\SO_{2l}}}
(\B_{\pi,\psi}-\B_{\pi^\prime,\psi})(twu)\tilde{f}_v(w_{l,l} twu, I_l).
\end{align*}
By Lemma \ref{l contribution to GL_l},
\begin{align*}
    &\quad\sum_{t\in T_{\SO_{2l}}, w\in \mathrm{B}_l(\SO_{2l}), u\in U_{\SO_{2l}}}
(\B_{\pi,\psi}-\B_{\pi^\prime,\psi})(twu)\tilde{f}_v(w_{l,l}twu, I_l)    \\
&\quad+\sum_{t\in T_{\SO_{2l}}, w\in \mathrm{B}_l^c(\SO_{2l}), u\in U_{\SO_{2l}}}
(\B_{\pi,\psi}-\B_{\pi^\prime,\psi})(twu)\tilde{f}_v(w_{l,l} twu, I_l) \\
&=\sum_{t\in T_{\SO_{2l}}, w\in \mathrm{B}_l(\SO_{2l}), u\in U_{\SO_{2l}}}
(\B_{\pi,\psi}-\B_{\pi^\prime,\psi}+\B_{c\cdot\pi,\psi}-\B_{c\cdot\pi^\prime,\psi})(twu)\tilde{f}_v(w_{l,l}twu, I_l).
\end{align*}
By Lemma \ref{l-1 contribution to GL_l},
\begin{align*}
&\quad\sum_{t\in T_l, w\in \mathrm{B}_{l-1}(\SO_{2l}), u\in U_{\SO_{2l}}}
(\B_{\pi,\psi}-\B_{\pi^\prime,\psi})(twu)\tilde{f}_v(w_{l,l} twu, I_l)\\
&=\sum_{t\in A_l, w\in \mathrm{B}_{l-1}(\SO_{2l}), u\in U_{\SO_{2l}}}
(\B_{\pi,\psi}-\B_{\pi^\prime,\psi}+\B_{c\cdot\pi,\psi}-\B_{c\cdot\pi^\prime,\psi})(twu)\tilde{f}_v(w_{l,l} twu, I_l).
\end{align*}
Therefore, we have that
\begin{align*}
0&=\sum_{t\in T_{\SO_{2l}}, w\in \mathrm{B}_l(\SO_{2l}), u\in U_{\SO_{2l}}}
(\B_{\pi,\psi}-\B_{\pi^\prime,\psi}+\B_{c\cdot\pi,\psi}-\B_{c\cdot\pi^\prime,\psi})(twu)\tilde{f}_v(w_{l,l}twu, I_l)\\
&\quad+\sum_{t\in A_l, w\in \mathrm{B}_{l-1}(\SO_{2l}), u\in U_{\SO_{2l}}}
(\B_{\pi,\psi}-\B_{\pi^\prime,\psi}+\B_{c\cdot\pi,\psi}-\B_{c\cdot\pi^\prime,\psi})(twu)\tilde{f}_v(w_{l,l} twu, I_l).
\end{align*}

Next, we define a function $f$ on a subset of $\GL_l$ so that we may apply Lemma $\ref{Nien}.$ The first step in this is to describe the arguments of the $W_v^*$'s. That is, we describe $a_i$, where the image of $tw$ under the embedding into $\SO_{2l+1}$ is $l_l(a_i)n_1 w_l n_2$ where $a_i\in\GL_l$ and $n_1,n_2\in V_l$ and $i=1,2$ if $w\in \mathrm{B}_{l}(\SO_{2l})$ or $w\in \mathrm{B}_{l-1}(\SO_{2l})$ respectively. We recall the setup. For $w\in\mathrm{B}_l^c(\SO_{2l})$, by Proposition \ref{l-1 Besselpart}, there exists $w^\prime\in W(\GL_l)$ such that $w=t_l(w^\prime)\tilde{w}_l$ where $w'=(w_{i,j}')_{i,j=1}^l\in W(\GL_l)$ with $w_{l,1}'=0.$ We also let $(w')^*=(w_{i,j}'^*)_{i,j=1}^l$ and $r$ be such that $w_{r,1}'=1.$ 

First, suppose that $l$ is even. Let 
$$
a_1^*=\left(\begin{matrix}
\frac{t_l\inv(w_{1,j}'^*)_{j=1}^{l-1}}{4} & \frac{1}{4} \\
\mathrm{diag}(t_{l-1}\inv,\dots, t_1\inv) (w_{i,j}'^*)_{i=2,j=1}^{i=l,j=l-1} & \mathrm{diag}(t_{l-1}\inv,\dots, t_1\inv)\frac{(w_{i,l}'^*)_{i=2}^l}{4} \\
\end{matrix}\right).
$$
Then,
$$
a_1^*=\left(\begin{matrix}
1 & & t_r & & \\
  & \ddots & & & \\
  & & 1 & & \\
  & & & \ddots & \\
  & & & & 1
\end{matrix}\right)
 \mathrm{diag}\left(\frac{t_l\inv}{4},t_{l-1}\inv,\dots, t_{r+1}\inv, \frac{t_r\inv}{4},t_{r-1}\inv,\dots,t_1\inv\right)
w'^*.
$$
where $t_r$ is the $(1,l-r+1)$ entry of the unipotent matrix. Thus, we obtain
$$
a_1=\left(\begin{matrix}
1 & &  & & \\
  & \ddots & & & \\
  & & 1 & & -t_r \\
  & & & \ddots & \\
  & & & & 1
\end{matrix}\right) \mathrm{diag}\left(t_1,\dots,t_{r-1}, 4t_r,t_{r+1},\dots,t_{l-1},4t_l \right) w',
$$
where $-t_r$ is the $(r,l)$ entry of the unipotent matrix. This determines $a$ on the $\mathrm{B}_{l}(\SO_{2l})$ sum. Let $t_{w'}\in T_{\GL_l}$ be such that $$t_{w'}\mathrm{diag}\left(t_1,\dots,t_{r-1}, 4t_r,t_{r+1},\dots,t_{l-1},4t_l \right)=\mathrm{diag}\left(t_1,\dots,t_{r-1}, t_r,t_{r+1},\dots,t_{l-1},t_l \right).$$ That is, $t_{w'}$ is a diagonal matrix consisting of $1$'s on the diagonal, except in the $(r,r)$ and $(l,l)$ coordinates where it is $\frac{1}{4}$.

Next, we consider the $a_2$ in the $\mathrm{B}_{l-1}(\SO_{2l})$ sum. Let $w^\prime=\left(\begin{matrix}
& w'' \\
1 & 
\end{matrix}\right)$ and
$$
a_2^*=\left(\begin{matrix}
& \frac{1}{4}+\frac{1}{4}t_l+\frac{1}{16}t_l\inv \\
\mathrm{diag}(t_{l-1}\inv,\dots,t_{1}\inv)(w'')^* &
\end{matrix}\right).
$$
Then,
$$
a_2^*=\left(\begin{matrix}
\frac{1}{4}+\frac{1}{4}t_l+\frac{1}{16}t_l\inv & & & \\
& t_{l-1}\inv & & \\
& & \ddots & \\
& & & t_{1}\inv 
\end{matrix}\right)
(w')^*.
$$
Hence
$$
a_2=\left(\begin{matrix}
t_1 & & & \\
& \ddots & & \\
& & t_{l-1} & \\
& & & (\frac{1}{4}+\frac{1}{4}t_l+\frac{1}{16}t_l\inv)\inv
\end{matrix}\right)
w'.
$$
Recall we partitioned $\mathbb{F}_q^\times\setminus\{\pm\frac{1}{2}\}$ into two disjoint sets $A$ and $B$ such that if $t_l\in A$ then $\frac{t_l\inv}{4}\in B.$ Suppose $t_l, s_l\in A$ with $\frac{1}{4}+\frac{1}{4}t_l+\frac{1}{16}t_l\inv=\frac{1}{4}+\frac{1}{4}s_l+\frac{1}{16}s_l\inv$. This gives a quadratic equation in $s_l$ whose roots are $s_l=t_l$ and $s_l=\frac{t_l\inv}{4}.$ Since $s,t \in A_l$ it follows that we must have $t_l=s_l.$ Let 
$$A_{GL_l}=\left\{\left(\begin{matrix}
t_1 & & & \\
& \ddots & & \\
& & t_{l-1} & \\
& & & (\frac{1}{4}+\frac{1}{4}t_l+\frac{1}{16}t_l\inv)\inv
\end{matrix}\right) \, | \, t_1,\dots,t_{l-1}\in\mathbb{F}_q^\times, t_l\in A \right\}.$$
The following map is well defined on $A_{GL_l}$: 

$$
\xi\left(\begin{matrix}
t_1 & & & \\
& \ddots & & \\
& & t_{l-1} & \\
& & & (\frac{1}{4}+\frac{1}{4}t_l+\frac{1}{16}t_l\inv)\inv
\end{matrix}\right)=\mathrm{diag}(t_1,\dots,t_l).
$$

Let $u=(u_{i,j})_{i,j=1}^l.$ Note that $u_{l,l+1}=0.$ Then, the embedding of $u$ in $\SO_{2l+1}$ is

$$
\dot{u}=\left(\begin{matrix}
(u_{i,j})_{i,j=1}^{i,j=l-1} & \left(\begin{matrix}
\frac{(u_{i,l})_{i=1}^{l-1}}{4}-\frac{(u_{i,l+1})_{i=1}^{l-1}}{2} & * & * \end{matrix}\right) & * \\
 & I_3 & * \\
 & & *
\end{matrix}\right).
$$
Let $\tilde{u}=\left(\begin{matrix}
(u_{i,j})_{i,j=1}^{i,j=l-1} &
\frac{(u_{i,l})_{i=1}^{l-1}}{4}-\frac{(u_{i,l+1})_{i=1}^{l-1}}{2} \\
0 & 1
\end{matrix}\right).$ Then $\dot{u}=l_l(\tilde{u})n_3.$ where $n_3\in V_l.$ The embedding takes $twu$ to $ l_l(a_i) n_1 w_l n_2 l_l(\tilde{u})n_3= n_4 l_l(a_i \tilde{u}^*) w_l n_5$ where $n_4, n_5\in V_l$ and $i=1,2$ if $w\in \mathrm{B}_{l}(\SO_{2l})$ or $w\in \mathrm{B}_{l-1}(\SO_{2l})$ respectively. Thus, by Proposition \ref{intertwine}, $\tilde{f}_v(w_{l,l} twu, I_l)=W_v^*(\mathrm{diag}(\frac{1}{2},\dots,\frac{1}{2})a_i \tilde{u}^*).$

Next, we define a function on a subset of $\GL_l$ using its Bruhat decomposition. Specifically, we partition the Weyl group of $W(\GL_l)$ into $W_1(\GL_l)$ and $W_2(\GL_l)$ where $w'\in W_1(\GL_l)$ if $w^\prime\neq\left(\begin{matrix}
& w'' \\
1 & 
\end{matrix}\right)$ for any $w''\in W(\GL_{l-1}).$  and $w'\in W_2(\GL_l)$ if $w^\prime=\left(\begin{matrix}
& w'' \\
1 & 
\end{matrix}\right)$ for some $w''\in W(\GL_{l-1}).$ By Proposition \ref{l-1 Besselpart}, we have $w=t_l(w')\tilde{w}_l\in \mathrm{B}_{l}(\SO_{2l})$ if $w'\in W_1(\GL_l)$ and $w=t_l(w')\tilde{w}_l\in \mathrm{B}_{l-1}(\SO_{2l})$ if $w'\in W_2(\GL_l).$ Let
$$
X_l=\left(\bigsqcup_{w'\in W_1(\GL_l)} U_{\GL_l}T_{\GL_l} w' U_{\GL_l}\right)
\bigsqcup
\left(\bigsqcup_{w'\in W_2(\GL_l)} U_{\GL_l}A_{\GL_l} w' U_{\GL_l}\right).
$$

Recall the definition of $t_{w'}$ above. For $g=u_1 t w' u_2\in U_{\GL_l}T_{\GL_l} w' U_{\GL_l}$ such that
$w'\in W_1(\GL_l)$ we define 
$$f(g)=
(\B_{\pi,\psi}-\B_{\pi^\prime,\psi}+\B_{c\cdot\pi,\psi}-\B_{c\cdot\pi^\prime,\psi})(t_l(\mathrm{diag}(2,\dots,2) t_l(u_1 t_{w'} t w' u_2)\tilde{w}_l).
$$ 
For $g=u_1 t w' u_2 \in U_{\GL_l}A_{\GL_l} w' U_{\GL_l}$ such that
$w'\in W_2(\GL_l)$ we define $$f(g)=
(\B_{\pi,\psi}-\B_{\pi^\prime,\psi}+\B_{c\cdot\pi,\psi}-\B_{c\cdot\pi^\prime,\psi})(t_l(\mathrm{diag}(2,\dots,2))t_l(u_1 \xi(t) w' u_2)\tilde{w}_l).
$$
We have $f(ug)=\psi(u)f(g)$ for any $u\in U_{\GL_l}$ and $g\in X_l.$ Also, $t_l(g)\tilde{w}_l=\tilde{w}_l t_l(g^*)$ for any $g\in\GL_l$ (we are still assuming that $l$ is even). Therefore, 
\begin{align*}
0&=\sum_{t\in T_{\GL_l}, w\in W_1(\GL_l), u\in U_{\SO_{2l}}}
f(\mathrm{diag}(\frac{1}{2},\dots,\frac{1}{2})tw\tilde{u}^*) W_v^*(\mathrm{diag}(\frac{1}{2},\dots,\frac{1}{2})tw\tilde{u}^*)\\
&\quad+\sum_{t\in A_{\GL_l}, w\in W_2(\GL_l), u\in U_{\SO_{2l}}}
f(\mathrm{diag}(\frac{1}{2},\dots,\frac{1}{2})tw\tilde{u}^*)W_v^*(\mathrm{diag}(\frac{1}{2},\dots,\frac{1}{2})tw \tilde{u}^*).
\end{align*}
Thus, by Lemma \ref{Nien}, $f$ must identically vanish on $X_l$. Therefore, for $l$ even, we have that $(\B_{\pi,\psi}+\B_{c\cdot\pi,\psi})(tw)=(\B_{\pi^\prime,\psi}+\B_{c\cdot\pi^\prime,\psi})(tw)$  for any $t\in\SO_{2l}$ and any $w\in \mathrm{B}_l(\SO_{2l})\cup\mathrm{B}_l^c(\SO_{2l}).$ Also, $(\B_{\pi,\psi}+\B_{c\cdot\pi,\psi})(tw)=(\B_{\pi^\prime,\psi}+\B_{c\cdot\pi^\prime,\psi})(tw)$ for any $w\in\mathrm{B}_{l-1}(\SO_{2l})$ and $t\in A_{l}$. Also, conjugating the previous equation by $\tilde{t}c$ gives $(\B_{\pi,\psi}+\B_{c\cdot\pi,\psi})(tw)=(\B_{\pi^\prime,\psi}+\B_{c\cdot\pi^\prime,\psi})(tw)$ for any $w\in\mathrm{B}_{l-1}(\SO_{2l})$ and $t\in B_{l}$ hence the theorem for $l$ even.

Next, suppose that $l$ is odd. Let 
$$
a_1^*=\left(\begin{matrix}
\frac{t_l\inv(w_{1,j}'^*)_{j=1}^{l-1}}{4} & \frac{1}{4} \\
\mathrm{diag}(t_{l-1}\inv,\dots, t_1\inv) (w_{i,j}'^*)_{i=2,j=1}^{i=l,j=l-1} & \mathrm{diag}(t_{l-1}\inv,\dots, t_1\inv)\frac{-(w_{i,l}'^*)_{i=2}^l}{2} \\
\end{matrix}\right).
$$
Then,
$$
a_1^*=\left(\begin{matrix}
1 & & \frac{-t_r}{2} & & \\
  & \ddots & & & \\
  & & 1 & & \\
  & & & \ddots & \\
  & & & & 1
\end{matrix}\right)
 \mathrm{diag}\left(\frac{t_l\inv}{4},t_{l-1}\inv,\dots, t_{r+1}\inv, \frac{-t_r\inv}{2},t_{r-1}\inv,\dots,t_1\inv\right)
w'^*,
$$
where $t_r$ is the $(1,l-r+1)$ entry of the unipotent matrix. Thus, we obtain
$$
a_1=\left(\begin{matrix}
1 & &  & & \\
  & \ddots & & & \\
  & & 1 & & \frac{t_r}{2} \\
  & & & \ddots & \\
  & & & & 1
\end{matrix}\right) \mathrm{diag}\left(t_1,\dots,t_{r-1}, -2t_r,t_{r+1},\dots,t_{l-1},4t_l \right) w',
$$
where $-t_r$ is the $(r,l)$ entry of the unipotent matrix. This determines $a$ on the $\mathrm{B}_{l}(\SO_{2l})$ sum. Let $t_{w'}\in T_{\GL_l}$ be such that $$t_{w'}\mathrm{diag}\left(t_1,\dots,t_{r-1}, -2t_r,t_{r+1},\dots,t_{l-1},4t_l \right)=\mathrm{diag}\left(t_1,\dots,t_{r-1}, t_r,t_{r+1},\dots,t_{l-1},t_l \right).$$ That is, $t_{w'}$ is a diagonal matrix consisting of $1$'s on the diagonal, except in the $(r,r)$ and $(l,l)$ coordinates where it is $\frac{-1}{2}$ and $\frac{1}{4}$ respectively.

Next, we consider the $a$ in the $\mathrm{B}_{l-1}(\SO_{2l})$ sum. Let $w^\prime=\left(\begin{matrix}
& w'' \\
1 & 
\end{matrix}\right)$ and
$$
a_2^*=\left(\begin{matrix}
& \frac{1}{2}(\frac{1}{2}-\frac{1}{4}(t_l+t_l\inv)) \\
\mathrm{diag}(t_{l-1}\inv,\dots,t_{1}\inv)(w'')^* &
\end{matrix}\right).
$$
Then,
$$
a_2^*=\left(\begin{matrix}
\frac{1}{2}(\frac{1}{2}-\frac{1}{4}(t_l+t_l\inv)) & & & \\
& t_{l-1}\inv & & \\
& & \ddots & \\
& & & t_{1}\inv 
\end{matrix}\right)
(w')^*.
$$
Hence
$$
a_2=\left(\begin{matrix}
t_1 & & & \\
& \ddots & & \\
& & t_{l-1} & \\
& & & (\frac{1}{2}(\frac{1}{2}-\frac{1}{4}(t_l+t_l\inv)))\inv
\end{matrix}\right)
w'.
$$
Recall we partitioned $\mathbb{F}_q^\times\setminus\{\pm\frac{1}{2}\}$ into two disjoint sets $A$ and $B$ such that if $t_l\in A$ then $t_l\inv\in B.$ Suppose $t_l, s_l\in A$ with $\frac{1}{2}(\frac{1}{2}-\frac{1}{4}(t_l+t_l\inv))=\frac{1}{2}(\frac{1}{2}-\frac{1}{4}(s_l+s_l\inv))$. This gives a quadratic equation in $s_l$ whose roots are $s_l=t_l$ and $s_l=t_l\inv.$ Since $s,t \in A_l$ it follows that we must have $t_l=s_l.$ Let 
$$A_{GL_l}=\left\{\left(\begin{matrix}
t_1 & & & \\
& \ddots & & \\
& & t_{l-1} & \\
& & & (\frac{1}{2}(\frac{1}{2}-\frac{1}{4}(t_l+t_l\inv)))\inv
\end{matrix}\right) \, | \, t_1,\dots,t_{l-1}\in\mathbb{F}_q^\times, t_l\in A \right\}.$$
The following map is well defined on $A_{GL_l}$: 

$$
\xi\left(\begin{matrix}
t_1 & & & \\
& \ddots & & \\
& & t_{l-1} & \\
& & & (\frac{1}{2}(\frac{1}{2}-\frac{1}{4}(t_l+t_l\inv)))\inv
\end{matrix}\right)=\mathrm{diag}(t_1,\dots,t_l).
$$

Let $u=(u_{i,j})_{i,j=1}^l.$ Note that $u_{l,l+1}=0.$ Then, the embedding of $u$ in $\SO_{2l+1}$ is

$$
\dot{u}=\left(\begin{matrix}
(u_{i,j})_{i,j=1}^{i,j=l-1} & \left(\begin{matrix}
\frac{(u_{i,l})_{i=1}^{l-1}}{4}-\frac{(u_{i,l+1})_{i=1}^{l-1}}{2} & * & * \end{matrix}\right) & * \\
 & I_3 & * \\
 & & *
\end{matrix}\right).
$$
Let $\tilde{u}=\left(\begin{matrix}
(u_{i,j})_{i,j=1}^{i,j=l-1} &
\frac{(u_{i,l})_{i=1}^{l-1}}{4}-\frac{(u_{i,l+1})_{i=1}^{l-1}}{2} \\
0 & 1
\end{matrix}\right).$ Then $\dot{u}=l_l(\tilde{u})n_3.$ where $n_3\in V_l.$ The embedding takes $twu$ to $ l_l(a_i) n_1 w_l n_2 l_l(\tilde{u})n_3= n_4 l_l(a_i \tilde{u}^*) w_l n_5$ where $n_4, n_5\in V_l$ and $i=1,2$ if $w\in \mathrm{B}_{l}(\SO_{2l})$ or $w\in \mathrm{B}_{l-1}(\SO_{2l})$ respectively. Thus, by Proposition \ref{intertwine}, $\tilde{f}_v(w_{l,l} twu, I_l)=W_v^*(\mathrm{diag}(\frac{1}{2},\dots,\frac{1}{2})a_i \tilde{u}^*).$

Next, we define a function on a subset of $\GL_l$ using its Bruhat decomposition. Specifically, we partition the Weyl group of $W(\GL_l)$ into two sets, $W_1(\GL_l)$ and $W_2(\GL_l)$, by  $w'\in W_1(\GL_l)$ if $w^\prime\neq\left(\begin{matrix}
& w'' \\
1 & 
\end{matrix}\right)$ for any $w''\in W(\GL_{l-1})$ and $w'\in W_2(\GL_l)$ if $w^\prime=\left(\begin{matrix}
& w'' \\
1 & 
\end{matrix}\right)$ for some $w''\in W(\GL_{l-1}).$ By Proposition \ref{l-1 Besselpart}, we have $w=t_l(w')\tilde{w}_l\in \mathrm{B}_{l}(\SO_{2l})$ if $w'\in W_1(\GL_l)$ and $w=t_l(w')\tilde{w}_l\in \mathrm{B}_{l-1}(\SO_{2l})$ if $w'\in W_2(\GL_l).$ Let
$$
X_l=\left(\bigsqcup_{w'\in W_1(\GL_l)} U_{\GL_l}T_{\GL_l} w' U_{\GL_l}\right)
\bigsqcup
\left(\bigsqcup_{w'\in W_2(\GL_l)} U_{\GL_l}A_{\GL_l} w' U_{\GL_l}\right).
$$

Recall the definition of $t_{w'}$ above. For $g=u_1 t w' u_2\in U_{\GL_l}T_{\GL_l} w' U_{\GL_l}$ such that
$w'\in W_1(\GL_l)$ we define 
$$f(g)=
(\B_{\pi,\psi}-\B_{\pi^\prime,\psi}+\B_{c\cdot\pi,\psi}-\B_{c\cdot\pi^\prime,\psi})(t_l(\mathrm{diag}(2,\dots,2) t_l(u_1 t_{w'} t w' u_2)\tilde{w}_l).
$$ 
For $g=u_1 t w' u_2 \in U_{\GL_l}A_{\GL_l} w' U_{\GL_l}$ with
$w'\in W_2(\GL_l)$ we define $$f(g)=
(\B_{\pi,\psi}-\B_{\pi^\prime,\psi}+\B_{c\cdot\pi,\psi}-\B_{c\cdot\pi^\prime,\psi})(t_l(\mathrm{diag}(2,\dots,2))t_l(u_1 \xi(t) w' u_2)\tilde{w}_l).
$$
We have $f(ug)=\psi(u)f(g)$ for any $u\in U_{\GL_l}$ and $g\in X_l.$ Let $g=(g_{i,j})_{i,j=1}^l$ and $g^*=(g^*_{i,j})_{i,j=1}^l$. Then $$
t_l(g)\tilde{w}_l=\tilde{w}_l
\left(\begin{matrix}
(g_{i,j})_{i,j=1}^{l-1} & 0 & (g_{i,l})_{i=1}^{l-1} & 0 \\
0 & g_{1,1}^* & 0 & (g_{1,j}^*)_{j=2}^{l} \\
(g_{l,j})_{j=1}^{l-1} & 0 & g_{l,l} & 0 \\
0 & (g_{i,1}^*)_{i=2}^{l} & 0 & (g_{i,j}^*)_{i,j=2}^{l} \\
\end{matrix}
\right).
$$
In particular, if $u\in U_{\GL_l}$, then 
$$
t_l(u)\tilde{w}_l=\tilde{w}_l
\left(\begin{matrix}
(g_{i,j})_{i,j=1}^{l-1} & 0 & (g_{i,l})_{i=1}^{l-1} & 0 \\
 & 1 & 0 & (g_{1,j}^*)_{j=2}^{l} \\
 &  & 1 & 0 \\
 & &  & (g_{i,j}^*)_{i,j=2}^{l} \\
\end{matrix}
\right),
$$
and the last matrix is upper triangular.
Therefore, 
\begin{align*}
0&=\sum_{t\in T_{\GL_l}, w\in W_1(\GL_l), u\in U_{\SO_{2l}}}
f(\mathrm{diag}(\frac{1}{2},\dots,\frac{1}{2})tw\tilde{u}^*) W_v^*(\mathrm{diag}(\frac{1}{2},\ldots,\frac{1}{2})tw\tilde{u}^*)\\
&\quad+\sum_{t\in A_{\GL_l}, w\in W_2(\GL_l), u\in U_{\SO_{2l}}}
f(\mathrm{diag}(\frac{1}{2},\dots,\frac{1}{2})tw\tilde{u}^*)W_v^*(\mathrm{diag}(\frac{1}{2},\dots,\frac{1}{2})tw \tilde{u}^*).
\end{align*}
Thus, by Lemma \ref{Nien}, $f$ must identically vanish on $X_l$. Therefore, for $l$ odd, we have that $(\B_{\pi,\psi}+\B_{c\cdot\pi,\psi})(tw)=(\B_{\pi^\prime,\psi}+\B_{c\cdot\pi^\prime,\psi})(tw)$  for any $t\in\SO_{2l}$ and $w\in \mathrm{B}_l(\SO_{2l}).$ Also, we have $(\B_{\pi,\psi}+\B_{c\cdot\pi,\psi})(tw)=(\B_{\pi^\prime,\psi}+\B_{c\cdot\pi^\prime,\psi})(tw)$ for any $w\in\mathrm{B}_{l-1}(\SO_{2l})$ and $t\in A_{l}$. Conjugation by $\tilde{t}c$ gives $(\B_{\pi,\psi}+\B_{c\cdot\pi,\psi})(tw)=(\B_{\pi^\prime,\psi}+\B_{c\cdot\pi^\prime,\psi})(tw)$ for any $w\in\mathrm{B}_{l-1}(\SO_{2l})$ and $t\in B_{l}$ for $l$ odd.

Finally, for any $l$, we have 
$(\B_{\pi,\psi}+\B_{c\cdot\pi,\psi})(tw)=(\B_{\pi^\prime,\psi}+\B_{c\cdot\pi^\prime,\psi})(tw)$  for any $t\in\SO_{2l}$ and $w\in \mathrm{B}_l(\SO_{2l}).$  Also, 
\begin{align*}
\B_{\pi,\psi}(tw)
&=\B_{\pi,\psi}(c\tilde{t}\inv \tilde{t} ctw c\tilde{t}\inv \tilde{t} c)\\
&=\B_{c\cdot\pi,\psi}( \tilde{t} ctw c\tilde{t}\inv)\\
&=\B_{c\cdot\pi,\psi}( t' cwc),
\end{align*}
where $t'=\tilde{t}ctc (cwc\tilde{t}\inv (cwc)\inv).$ Hence  for any $t\in\SO_{2l}$ and $w\in \mathrm{B}_l^c(\SO_{2l}),$ we have that $(\B_{\pi,\psi}+\B_{c\cdot\pi,\psi})(tw)=(\B_{\pi^\prime,\psi}+\B_{c\cdot\pi^\prime,\psi})(tw)$ and thus we have proved the theorem.
 \end{proof}

The following corollary, when combined with Corollary \ref{conj n gamma}, shows that $\gamma$-factor is unable to distinguish between a representation and its conjugate. That is, $\gamma(\pi\times\tau,\psi)=\gamma(c\cdot\pi\times\tau,\psi)$ for all irreducible generic representations $\tau$ of $\GL_{n}$ with $n\leq l$.

\begin{cor}\label{conj l gamma}
Let $\pi$ be an irreducible cuspidal $\psi$-generic representation of $\SO_{2l}$. Then we have $\gamma(\pi\times\tau,\psi)=\gamma(c\cdot\pi\times\tau,\psi)$ for all irreducible generic representations $\tau$ of $\GL_{l}$.
\end{cor}

\begin{proof}
By Corollary \ref{conj n gamma}, $\gamma(\pi\times\tau,\psi)=\gamma(c\cdot\pi\times\tau,\psi)$ for all irreducible generic representations $\tau$ of $\GL_{n}$ with $n\leq l-1$. So, by Theorems \ref{GL_n} and \ref{GL_{l-1}}, 
\begin{align*}
&\quad\Psi(\B_{\pi,\psi},\tilde{f}_v)-\Psi(\B_{c\cdot\pi,\psi},\tilde{f}_v) \\
&=\sum_{t\in T_{\SO_{2l}}, w\in \mathrm{B}_l(\SO_{2l}), u\in U_{\SO_{2l}}}
(\B_{\pi,\psi}-\B_{c\cdot\pi,\psi})(twu)\tilde{f}_v(w_{l,l}twu, I_l)\\
&\quad+\sum_{t\in T_{\SO_{2l}}, w\in \mathrm{B}_l^c(\SO_{2l}), u\in U_{\SO_{2l}}}
(\B_{\pi,\psi}-\B_{c\cdot\pi,\psi})(twu)\tilde{f}_v(w_{l,l} twu, I_l)\\
&\quad+\sum_{t\in T_l, w\in \mathrm{B}_{l-1}(\SO_{2l}), u\in U_{\SO_{2l}}}
(\B_{\pi,\psi}-\B_{c\cdot\pi,\psi})(twu)\tilde{f}_v(w_{l,l} twu, I_l).
\end{align*}
Performing the change of variables as in the proof of Theorem \ref{GL_l} on the $\B_{\pi,\psi}$ sums gives $\Psi(\B_{\pi,\psi},\tilde{f}_v)-\Psi(\B_{c\cdot\pi,\psi},\tilde{f}_v)=0$. By Proposition \ref{GL_lnonzero}, we may choose a nonzero $v\in\tau$ such that $\Psi(\B_{\pi,\psi},f_v)=\Psi(\B_{c\cdot\pi,\psi},f_v)=W_v(w_{l,l})\neq 0.$ Thus, we have $\gamma(\pi\times\tau,\psi)=\gamma(c\cdot\pi\times\tau,\psi)$. This proves the corollary.
 \end{proof}

\section{The converse theorem}

In this section, we prove the converse theorem. First, we combine the results of the previous sections to obtain the following theorem.

\begin{thm}\label{Bessels Equal}
Let $\pi$ and $\pi^\prime$ be irreducible cuspidal $\psi$-generic representations of split $\SO_{2l}(\mathbb{F}_q)$ with the same central character. If $$\gamma(\pi\times\tau,\psi)=\gamma(\pi^\prime\times\tau,\psi),$$
for all irreducible generic representations $\tau$ of $\GL_n$ with $n\leq l,$
then we have that $$(\B_{\pi,\psi}+\B_{c\cdot\pi,\psi})(g)=(\B_{\pi^\prime,\psi}+\B_{c\cdot\pi^\prime,\psi})(g)$$  for any $g\in\SO_{2l}(\mathbb{F}_q)$.
\end{thm}

\begin{proof}
By the Bruhat decomposition, we may assume $g=u_1 t w u_2\in B_{\SO_{2l}}wB_{\SO_{2l}}$ for some $w\in W(\SO_{2l}).$ By Proposition \ref{Besselprop} and the definition of the Bessel support, it is enough to show that $(\B_{\pi,\psi}+\B_{c\cdot\pi,\psi})(tw)=(\B_{\pi^\prime,\psi}+\B_{c\cdot\pi^\prime,\psi})(tw)$ for any $t\in T_{\SO_{2l}}$ and $w\in \mathrm{B}(\SO_{2l}).$ This follows from Lemma \ref{center} (we use that $\pi$ and $\pi'$ have the same central character here), Proposition \ref{BesselPartition}, Theorems \ref{GL_n}, \ref{GL_{l-1}}, \ref{GL_l}, and Corollary \ref{cGL_n}. This concludes the proof of the theorem.
 \end{proof}

\begin{thm}[The Converse Theorem for $\SO_{2l}$]\label{converse thm}
Let $\pi$ and $\pi^\prime$ be irreducible cuspidal $\psi$-generic representations of split $\SO_{2l}(\mathbb{F}_q)$ with the same central character. If $$\gamma(\pi\times\tau,\psi)=\gamma(\pi^\prime\times\tau,\psi),$$ 
for all irreducible generic representations $\tau$ of $\GL_n(\mathbb{F}_q)$ with $n\leq l,$
then $\pi\cong\pi'$ or $\pi\cong c\cdot\pi'.$
\end{thm}

\begin{proof}
By Theorem \ref{Bessels Equal}, $(\B_{\pi,\psi}+\B_{c\cdot\pi,\psi})=(\B_{\pi^\prime,\psi}+\B_{c\cdot\pi^\prime,\psi})$ on all of $\SO_{2l}.$ We let $W_\pi, W_{c\cdot\pi}, W_{\pi^\prime},$ and $W_{c\cdot\pi^\prime}$ be the Whittaker models of $\pi, c\cdot\pi, \pi^\prime$, and $c\cdot\pi^\prime$ respectively.

First, suppose that $\pi$ is isomorphic to $c\cdot\pi$. By uniqueness of Whittaker models, we have $\B_{\pi,\psi}+\B_{c\cdot\pi,\psi}=2\B_{\pi,\psi} \in W_\pi$. Then, $2\B_{\pi,\psi}=\B_{\pi^\prime,\psi}+\B_{c\cdot\pi^\prime,\psi}$ and hence $W_\pi \cap W_{\pi^\prime} \oplus W_{c\cdot\pi^\prime}\neq 0.$ Since $W_\pi$ is isomorphic to $\pi$ and is hence irreducible, we have that $W_\pi \cap W_{\pi^\prime} \oplus W_{c\cdot\pi^\prime}=W_\pi$ and therefore, by uniqueness of Whittaker models, $W_\pi$ must be isomorphic to one of  $W_{\pi^\prime}$ or $W_{c\cdot\pi^\prime}$.
 
Second, suppose $\pi$ is not isomorphic to $c\cdot\pi$. We have that $\B_{\pi,\psi}\in W_\pi$. We also have that $\B_{\pi^\prime,\psi}+\B_{c\cdot\pi^\prime,\psi}-\B_{c\cdot\pi,\psi}\in W_{\pi^\prime} \oplus W_{c\cdot\pi^\prime}\oplus W_{c\cdot\pi}$. Thus, $W_\pi \cap W_{\pi^\prime} \oplus W_{c\cdot\pi^\prime}\oplus W_{c\cdot\pi}\neq 0$ and hence the intersection is a nonzero subrepresentation of $W_\pi$. Since $W_\pi$ is isomorphic to $\pi$ and is hence irreducible, we have that $W_\pi \cap W_{\pi^\prime} \oplus W_{c\cdot\pi^\prime}\oplus W_{c\cdot\pi} =W_\pi$ and therefore, by uniqueness of Whittaker models, $W_\pi$ must be isomorphic to one of  $W_{c\cdot\pi}, W_{\pi^\prime},$ or $W_{c\cdot\pi^\prime}$. By assumption, $W_\pi$ is not isomorphic to $W_{c\cdot\pi}$ and hence must be isomorphic to $W_{\pi^\prime}$ or $W_{c\cdot\pi^\prime}$. 

Therefore, in either case, we have shown the converse theorem.
 \end{proof}

\end{document}